\documentclass[a4paper,11pt,reqno,noindent]{amsart}
\usepackage[centertags]{amsmath}
\usepackage{amsfonts,amssymb,amsthm,dsfont,cases,amscd,esint,enumerate}
\usepackage[T1]{fontenc}
\usepackage[english]{babel}
\usepackage[applemac]{inputenc}
\usepackage{newlfont}
\usepackage{color}
\usepackage[body={15cm,21.5cm},centering]{geometry} 
\usepackage{fancyhdr}
\usepackage{enumerate}

\theoremstyle{plain}
\newtheorem{theor10}{Theorem}
\newenvironment{theor1}
  {\pushQED{\qed}\begin{theor10}}
  {\popQED\end{theor10}}
\newtheorem{prop10}{Proposition}

\newtheorem{cor10}{Corollary}

\newtheorem{theor0}{Theorem}[section]

\newtheorem{lem0}[theor0]{Lemma}
\newenvironment{lem}
  {\pushQED{\qed}\begin{lem0}}
  {\popQED\end{lem0}}
\newtheorem{prop0}[theor0]{Proposition}
\newenvironment{prop}
  {\pushQED{\qed}\begin{prop0}}
  {\popQED\end{prop0}}
\newtheorem{cor0}[theor0]{Corollary}
\newenvironment{cor}
  {\pushQED{\qed}\begin{cor0}}
  {\popQED\end{cor0}}
\newtheorem{propr0}[theor0]{Property}

\newtheorem{hyp0}[theor0]{Hypothesis}

\newtheorem{result0}[theor0]{Result}

\newtheorem{conj0}[theor0]{Conjecture}

\newtheorem{heur0}[theor0]{Heuristics}

\theoremstyle{definition}
\newtheorem{defin0}[theor0]{Definition}
\newenvironment{defin}
  {\pushQED{\qed}\begin{defin0}}
  {\popQED\end{defin0}}
\newtheorem{rems0}[theor0]{Remarks}

\newtheorem{ex0}[theor0]{Example}

\newtheorem{exs0}[theor0]{Examples}

\newtheorem{rem0}[theor0]{Remark}

\newtheorem{qu0}[theor0]{Question}

\newtheorem{qus0}[theor0]{Questions}

  \newtheorem{as0}[theor0]{Assumption}

\numberwithin{equation}{section}
\mathchardef\emptyset="001F

\numberwithin{equation}{section}
\newcommand{\N}{\mathbb N}
\newcommand{\Sp}{\mathbb S}
\newcommand{\e}{\varepsilon}
\newcommand{\Log}{|\!\log\e|}

\newcommand{\calQ}{\mathcal{Q}}
\newcommand{\calP}{\mathcal{P}}

\newcommand{\Mcal}{\mathcal{M}}
\newcommand{\Lc}{\mathcal{L}}

\newcommand{\Dm}{\mathbb{D}}

\newcommand{\Kc}{\mathcal K}
\newcommand{\Sc}{\mathcal S}
\newcommand{\Rc}{\mathcal R}

\newcommand{\Hf}{\mathfrak H}

\newcommand{\R}{\mathbb R}

\newcommand{\F}{\mathcal F}
\newcommand{\Nc}{\mathcal N}

\newcommand{\Id}{\operatorname{Id}}
\newcommand{\op}{{\operatorname{op}}}
\newcommand{\E}{\mathbb{E}}

\newcommand{\Cov}{\operatorname{Cov}}

\newcommand{\ee}{e}
\newcommand{\Aa}{\boldsymbol a}
\newcommand{\bb}{{\boldsymbol b}}

\newcommand{\Ld}{\operatorname{L}}

\newcommand{\sym}{{\operatorname{sym}}}

\newcommand{\step}[1]{\noindent \textit{Step} #1.}

\newcommand{\Pm}{\mathbb{P}}
\newcommand{\pr}[1]{\mathbb{P}\left[ #1 \right]}

\newcommand{\expec}[1]{\mathbb{E}\left[ #1 \right]}
\newcommand{\expecm}[1]{\mathbb{E}\big[ #1 \big]}
\newcommand{\var}[1]{\mathrm{Var}\left[#1\right]}

\newcommand{\cov}[2]{\operatorname{Cov}\left[{#1};{#2}\right]}

\newcommand{\expeC}[2]{\mathbb{E}\left[\left. #1 \,\right\|\,#2\right]}

\newcommand{\dTV}[2]{\operatorname{d}_{\operatorname{TV}}\left({#1};{#2}\right)}

\newcommand{\dWW}[2]{\operatorname{W}_2\left({#1};{#2}\right)}

\usepackage{color}
\usepackage[colorlinks,citecolor=black,urlcolor=black]{hyperref}

\title[Scaling limit of the homogenization commutator]{Scaling limit of the homogenization commutator\\for Gaussian coefficient fields}

\author[M. Duerinckx]{Mitia Duerinckx}
\author[J. Fischer]{Julian Fischer}
\author[A. Gloria]{Antoine Gloria}
\address[Mitia Duerinckx]{Laboratoire de Mathématique d'Orsay, UMR 8628, Université Paris-Sud, F-91405 Orsay, France \& Universit\'e Libre de Bruxelles, Département de Mathématique, Brussels, Belgium}
\email{mduerinc@ulb.ac.be}
\address[Julian Fischer]{Institute of Science and Technology Austria (IST Austria), Am Campus~1, 3400 Klosterneuburg, Austria}
\email{julian.fischer@ist.ac.at}
\address[Antoine Gloria]{Sorbonne Universit\'e, CNRS, Universit\'e de Paris, Laboratoire Jacques-Louis Lions (LJLL), F-75005 Paris, France \& Universit\'e Libre de Bruxelles, Département de Mathématique, Brussels, Belgium}
\email{gloria@ljll.math.upmc.fr}

\date{}

\begin{document}
\selectlanguage{english}

\begin{abstract}
Consider a linear elliptic partial differential equation in divergence form with a random coefficient field. The solution-operator displays fluctuations around its expectation. The recently-developed pathwise theory of fluctuations in stochastic homogenization reduces the characterization of these fluctuations to those of the so-called standard homogenization commutator. In this contribution, we investigate the scaling limit of this key quantity: starting from a Gaussian-like coefficient field with possibly strong correlations,
we establish the convergence of the rescaled commutator to a fractional Gaussian field, depending on the decay of correlations of the coefficient field, and we investigate the (non)degeneracy of the limit. This extends to general dimension $d\ge 1$ previous results so far limited to dimension $d=1$, and to the continuum setting with strong correlations 
recent results in the discrete iid case.

\medskip

\end{abstract}

\maketitle
\setcounter{tocdepth}{1}
\tableofcontents
\bigskip

\section{Introduction}
\subsection{General overview}
Let $\Aa$ be a stationary and ergodic random coefficient field that satisfies the boundedness and ellipticity properties
\begin{equation}\label{f.56}
|\Aa(x)\xi|\le|\xi|,\qquad\xi\cdot\Aa(x)\xi\ge\lambda|\xi|^2,\qquad\mbox{for all}\;x,\xi\in\R^d,
\end{equation}
for some $\lambda>0$.
Given a deterministic vector field $f\in C^\infty_c(\R^d)^d$, we consider the random family $(\nabla u_\e)_{\e>0}$ of unique Lax-Milgram solutions (which henceforth means the unique weak solutions in $\dot H^1(\R^d)$) to the following rescaled elliptic equations in $\R^d$,
\begin{equation}\label{eq:first-def-ups}
-\nabla\cdot\big(\Aa(\tfrac\cdot\e)\nabla u_\e\big)\,=\,\nabla\cdot f.
\end{equation}
It is known since the pioneering work of Papanicolaou and Varadhan~\cite{PapaVara} and of Kozlov~\cite{Kozlov-79} that, almost surely, $\nabla u_\e$ converges weakly in $\Ld^2(\R^d)$ as $\e\downarrow0$ to the unique Lax-Milgram solution $\bar u$ in $\R^d$ of
\[-\nabla\cdot\big(\bar\Aa\nabla\bar u)=\nabla\cdot f,\]
where $\bar\Aa$ is a deterministic and constant matrix that only depends on the law of $\Aa$ and is given for $1\le i\le d$ by
\[\bar\Aa\ee_i=\expec{\Aa(\nabla\phi_i+\ee_i)},\]
in terms of the so-called corrector $\phi_i$ in the direction $\ee_i$ (cf.~Lemma~\ref{si} below).
Most results on quantitative stochastic homogenization in the last decade focused on the accurate description of the spatial oscillations of the solution operator for~\eqref{eq:first-def-ups} (e.g.~\cite{GO1,GO2,GNO1}, \cite{AS,GNO-reg,AKM-book}, and  the references therein).
In this contribution we rather focus on the random fluctuations of macroscopic observables of the form $\int_{\R^d}g\cdot\nabla u_\e$ or $\int_{\R^d}g\cdot\Aa(\tfrac\cdot\e)\nabla u_\e$ with $g\in C^\infty_c(\R^d)^d$, and establish (quantitative) central limit theorems.
More precisely, pursuing the investigation of our previous works on the topic~\cite{DGO1,DGO2,DO1} (see also~\cite{MO,GN,GuM,MN,AKM2}), and inspired by previous computations in the one-dimensional setting~\cite{BGMP-08,Gu-Bal-12,LNZH-17}, the present contribution aims at analyzing the effects of strong correlations of the coefficient field $\Aa$. For simplicity and concreteness, we focus on the following Gaussian model family of coefficient fields.
This particular setting leads to significant simplifications since  Malliavin calculus then allows to systematically linearize the dependence on the randomness.
\begin{defin}\label{def:coeff}
The coefficient field $\Aa$ is said to be \emph{Gaussian with parameter~$\beta>0$} if it has the form
\[\Aa(x)\,:=\,a_0(G(x)),\]
where $a_0\in C^2_b(\R^\kappa)^{d\times d}$ is such that the boundedness and ellipticity assumptions~\eqref{f.56} are satisfied pointwise, and where $G$ is some $\R^\kappa$-valued centered stationary Gaussian random field on $\R^d$ constructed on a probability space $(\Omega,\F,\Pm)$ (with expectation $\E$), characterized by its covariance function
\[c(x)\,:=\,\expec{G(x)\otimes G(0)},\qquad c:\R^d\to\R^{\kappa\times\kappa},\]
which is assumed to have $\beta$-algebraic decay at infinity in the following sense: there exists $C_0>0$ such that for all $x\in\R^d$,
\begin{equation}
\tfrac1{C_0}(1+|x|)^{-\beta}\,\le\,|c(x)| \, \le \,C_0(1+|x|)^{-\beta}, \label{eq:cov-as}
\end{equation}
and in the case $\beta<d$ we further assume $|\nabla c(x)|\le C_0(1+|x|)^{-\beta-1}$.
In addition, we assume that $c$ can be decomposed as $c=c_0\ast c_0$ where $c_0$ satisfies\footnote{Note that this decay assumption for $c_0$ (including the logarithmic correction in the critical case $\beta=d$) precisely ensures that $c=c_0\ast c_0$ satisfies the upper bound in~\eqref{eq:cov-as}.}
\begin{equation}
|c_0(x)|\,\le\,C_0(1+|x|)^{-\frac12(d+\beta)} \times\left\{ 
\begin{array}{lcl}
1&:&\beta \ne d,
\\
\log^{-\frac12} (1+|x|)&:&\beta=d.
\end{array}
\right. \label{eq:cov-as-bis}\qedhere
\end{equation}
\end{defin}

\noindent
Since the covariance function $c$ decays at infinity, the Gaussian random field~$G$ is known to be (strongly) mixing.
In particular, $G$ is ergodic, which ensures existence and uniqueness of correctors and homogenized coefficients (cf.~Lemma~\ref{si}).
Note however that $G$ is $\alpha$-mixing only if the covariance is integrable, that is, if $\beta>d$ (e.g.~\cite{Doukhan-94}).

\medskip
In the companion articles~\cite{DGO1,DGO2}, it was shown that fluctuations of macroscopic observables are determined at leading order by those of the so-called standard homogenization commutator (see also~\cite{AS,AKM2})
\[\Xi_i\,:=\,(\Aa-\bar\Aa)(\nabla\phi_i+\ee_i).\]
This is referred to as the \emph{pathwise structure of fluctuations} in stochastic homogenization, which originates in the crucial observation that the $2$-scale expansion of commutators remains accurate in the fluctuation scaling.
More precisely,
the results in~\cite{DGO2} take on the following guise, where all scalings and rates are (generically) optimal. Henceforth, we focus on dimensions $d>1$ --- the one-dimensional setting is indeed much simpler since equation~\eqref{eq:first-def-ups} can then be explicitly integrated.
\begin{enumerate}[$\bullet$]
\item \emph{Fluctuation scaling:}
For all $f,g\in C^\infty_c(\R^d)^d$ and $p<\infty$,
\[\quad\expec{\Big|\pi_{d,\beta}(\tfrac1\e)^\frac12\int_{\R^d}g\cdot\nabla u_\e\Big|^p}^\frac1p\,\lesssim_{p,f,g}\,1,\]
where the rescaling is defined by
\begin{align}\label{e.pi*}
\quad\pi_{d,\beta}(r)\,:=\,
\left\{\begin{array}{lll}
(1+r)^d&:&\beta>d,\\
\frac{(1+r)^d}{\log(2+r)}&:&\beta=d,\\
(1+r)^\beta&:&\beta<d.
\end{array}\right.
\end{align}
\item \emph{Pathwise structure of fluctuations:}
For all $f,g\in C^\infty_c(\R^d)^d$ and $p<\infty$,
\begin{multline}\label{eq:pathwise}
\qquad\expec{\Big|\pi_{d,\beta}(\tfrac1\e)^\frac12\Big(\int_{\R^d} g\cdot \nabla\big(u_\e-\expec{u_\e}\big)+\int_{\R^d}(\bar\calP_H^*g)\cdot\Xi_i(\tfrac\cdot\e)\nabla_i\bar u\Big)\Big|^p}^\frac1p\\
\quad+\expec{\Big|\pi_{d,\beta}(\tfrac1\e)^\frac12\Big(\int_{\R^d} g\cdot\big(\Aa(\tfrac\cdot\e)\nabla u_\e-\expecm{\Aa(\tfrac\cdot\e)\nabla u_\e}\big)-\int_{\R^d}(\bar\calP_L^*g)\cdot\Xi_i(\tfrac\cdot\e)\nabla_i\bar u\Big)\Big|^p}^\frac1p\\
\,\lesssim_{p,f,g}\,\e\mu_{d,\beta}(\tfrac1\e),
\end{multline}
in terms of the homogenized Helmholtz and Leray projections on $\Ld^2(\R^d)^d$,
\[\quad\bar\calP_H^*:=\nabla(\nabla\cdot\bar\Aa^* \nabla)^{-1}\nabla\cdot,\qquad \bar\calP_L^*:=\Id-\bar\calP_H\bar\Aa^*,\]
where we have set
\begin{equation}\label{mudbeta}
\quad\mu_{d,\beta}(r)\,:=\, 
\left\{
\begin{array}{lll}
1 &:&\beta>2,\,d>2,\\
\log^\frac12(2+r)&:&\beta>2,\,d=2,~\text{or}~\beta=2,\,d>2,\\
\log(2+r)&:&\beta=2,\,d=2,\\
(1+r)^{1-\frac\beta2}&:&\beta<2,\,d\ge 2.
\end{array}
\right.
\end{equation}
\end{enumerate}
These results reduce the description of fluctuations of macroscopic observables at leading order to the fluctuations of (large-scale averages of) the standard homogenization commutator $\Xi$ only. In order to fully describe fluctuations of macroscopic observables, it then remains to analyze the scaling limit of $\Xi$ itself. Under strong decay assumptions on the correlations of the coefficient field, the rescaled commutator $\e^{-\frac d2}\Xi(\tfrac\cdot\e)$ is known to converge in law (as a random Schwartz distribution) to a Gaussian white noise, which was first established in the discrete setting in~\cite{DGO1}, in the case of finite range of dependence in~\cite{AKM2,GO4}, and in the integrable Gaussian setting ($\beta>d$) in~\cite{DO1}.
In the present contribution, we analyze the corresponding scaling limit for the whole Gaussian family of coefficient fields, including sharp convergence rates, and we emphasize the effects of strong correlations.

\subsection{Main results}
We address two main questions:
\begin{itemize}
\item The scaling limit of the commutator, both qualitatively and quantitatively, for weak and strong correlations;
\item The non-degeneracy of the scaling limit.
\end{itemize}
Before we state the main results, let us emphasize that this analysis is possible because the key object for fluctuations in stochastic homogenization, the homogenization commutator, turns out to be a local map of the coefficients. This appears clearly in \cite{AKM2,GO4} in the case of an ensemble of finite range of dependence, where it is proved 
that the homogenization commutator is also a locally-dependent random field. The proof strongly relies on the fact that the mixing condition is linear (in the sense it is compatible with renormalization techniques, or iterations). In the present article, we consider Gaussian coefficients, for which mixing conditions (in form of functional inequalities) are nonlinear (in particular, these are not easily iterated). In this setting the locality of the homogenization commutator is a nonlinear one, more in the spirit of \cite{DGO1}. 
As opposed to  \cite{AKM2,GO4}, the upcoming results are not only qualitative, but also quantitative.

\medskip

The following states that in the Gaussian setting the scaling limit of the standard homogenization commutator is a Gaussian white noise whenever correlations are integrable, that is, whenever $\beta\ge d$, while in the non-integrable case $\beta<d$ the scaling limit is a fractional Gaussian field. This illustrates that the locality property of the commutator 
with respect to the coefficients is a \emph{relative locality}. This fully extends to the multidimensional setting the (explicit) computations of~\cite{BGMP-08} for $d=1$, and
extends the results of \cite{DGO1} in the iid discrete case to this continuum setting with correlations. Finer statements for the convergence of the covariance structure with optimal rates are included in Section~\ref{chap:conv-cov}, cf.~Proposition~\ref{prop:conv-cov}, and are completely new (even for integrable correlations).
To ease the reading, only a simplified version of these resuls is given below. 

\begin{theor1}\label{th:main}
Let the coefficient field $\Aa$ be Gaussian with parameter $\beta>0$ as in Definition~\ref{def:coeff}.
For $F\in C^\infty_c(\R^d)^{d\times d}$, we write for short
\[I_\e(F)\,:=\,\pi_{d,\beta}(\tfrac1\e)^{\frac12}\int_{\R^d}F(x):\Xi(\tfrac x\e)\,dx.\]
\begin{enumerate}[(i)]
\item \emph{Convergence of the covariance structure:}
\begin{enumerate}[$\bullet$]
\smallskip\item \emph{Integrable case $\beta>d$:} There exists a constant tensor $\calQ$ of order $4$ such that for all $F,F'\in C^\infty_c(\R^d)^{d\times d}$,
\[\qquad\qquad\lim_{\e\downarrow0}\cov{I_\e(F)}{I_\e(F')}\,=\,\int_{\R^d}F(x):\calQ:F'(x)\,dx.\]
\item \emph{Critical case $\beta=d$:} If for all $x$ the rescaled covariance $L^dc(Lx)$ admits a limit as $L\uparrow\infty$, then the same conclusion holds as in the integrable case.
\smallskip\item \emph{Non-integrable case $\beta<d$:} If for all $x$ the rescaled covariance $L^\beta c(Lx)$ admits a limit as $L\uparrow\infty$, then there exists a $4$th-order tensor field $\calQ$ on $\Sp^{d-1}$ such that for all $F,F'\in C^\infty_c(\R^d)^{d\times d}$,
\[\qquad\qquad\lim_{\e\downarrow0}\cov{I_\e(F)}{I_\e(F')}\,=\,\int_{\R^d}\int_{\R^d}F(x):\frac{\calQ(\frac{x-y}{|x-y|})}{|x-y|^\beta}:F'(y)\,dxdy.\]
\end{enumerate}
\smallskip\item \emph{Asymptotic normality:}
For all $F\in C^\infty_c(\R^d)^{d\times d}$ and $\e>0$,  
\begin{multline*}
\qquad\dWW{\frac{I_\e(F)}{\var{I_\e(F)}^\frac12}}\Nc+\dTV{\frac{I_\e(F)}{\var{I_\e(F)}^\frac12}}\Nc\\
\,\lesssim_F\,\frac1{\var{I_\e(F)}}\left\{\begin{array}{lll}
\e^\frac d2\Log&:&\beta>d,\\
\e^\frac d2 \Log^\frac32\log\Log  &:& \beta=d,\\
\e^\frac\beta2&:&\beta<d,
\end{array}\right.
\end{multline*}
where $\dWW\cdot\Nc$ and $\dTV\cdot\Nc$ denote the $2$-Wasserstein (see e.g.~\cite{NP-book}) and the total variation distance to a standard Gaussian law, respectively.
\end{enumerate}
In particular, if the limiting covariance structure is non-degenerate, that is, if for all nonzero test functions $F\in C^\infty_c(\R^d)^{d\times d}$, $\liminf_{\e}\var{I_\e(F)}>0$, and further assuming in the non-integrable case $\beta\le d$ that the rescaled covariance $L^{\beta}c(L\cdot)$ admits a pointwise limit as $L\uparrow\infty$, then the rescaled homogenization commutator $\pi_{d,\beta}(\frac1\e)^\frac12\Xi(\frac\cdot\e)$ converges in law (as a random Schwartz distribution) to a (matrix-valued) Gaussian white noise with variance $\calQ$ in the integrable case $\beta\ge d$, or to a (matrix-valued) fractional Gaussian field with kernel $\calQ(\frac{x}{|x|})|x|^{-\beta}$ in the non-integrable case $\beta<d$.
\qedhere
\end{theor1}
The additional condition on the convergence of the rescaled covariance of~$G$ in the non-integrable case is necessary: strong oscillations of the covariance of~$G$ can break down the convergence of the covariance structure of~$\Xi$ (it suffices to consider rescaled covariances $L^\beta c(Lx)$ with several cluster points when $L\uparrow \infty$) . This is a new feature due to strong correlations. Likewise, convergence rates can be arbitrarily slow.
The proof follows the general structure of the analysis of the i.i.d.\@ discrete case in~\cite{DGO1} and makes strong use of tools from Malliavin calculus as in~\cite{DO1}.

\medskip

Combining this result with the pathwise structure of fluctuations~\eqref{eq:pathwise}, we are led to a quantitative CLT (with optimal rates) for all macroscopic observables.
An important question concerns the possible degeneracy of the limit: as observed for $d=1$ in~\cite{Gu-Bal-12,LNZH-17} (see also~\cite{Taqqu}), degeneracy may occur and leads to different, non-Gaussian behaviors.
In Section~\ref{chap:deg}, we establish the following sufficient criteria, based on the explicit characterization of the limiting covariance structures provided by the Malliavin approach. Note that the condition in the non-integrable case is much more restrictive than in the integrable case.
\begin{enumerate}[$\bullet$]
\item In the integrable case $\beta>d$, if $\Aa=a_0(G)$ is symmetric, if there exist $y,\alpha \in \R^\kappa$ such that the symmetric matrix $\alpha_l \partial_la_0(y)$ is definite, and if the covariance function $c$ is smooth at the origin, then the fluctuation tensor $\mathcal Q$ is non-degenerate.
\smallskip\item In the non-integrable case $\beta<d$, if $\Aa=a_0(G)$ is symmetric and if for some $1\le l\le\kappa$ the symmetric matrix $\partial_la_0(y)$ is definite for all $y\in \R^\kappa$, then the fluctuation tensor field~$\mathcal Q$ is non-degenerate. Many degenerate examples can however be constructed.
\smallskip\item In both the integrable and the non-integrable cases, \emph{non-degeneracy is generic}.
\end{enumerate}
Precise statements are postponed to Section~\ref{chap:deg}.

\bigskip\noindent
\begin{samepage}{\bf Notation}
\begin{enumerate}[$\bullet$]
\item We denote by $C\ge1$ any constant that only depends on $d$, $\lambda$, $\|a_0\|_{W^{2,\infty}}$, and on the covariance function $c$ via the constants $C_0,\beta$ in~\eqref{eq:cov-as} \&~\eqref{eq:cov-as-bis}. We use the notation $\lesssim$ (resp.~$\gtrsim$) for $\le C\times$ (resp.\@ $\ge\frac1C\times$) up to such a multiplicative constant $C$. We write $\simeq$ when both $\lesssim$ and $\gtrsim$ hold. We add subscripts to $C,\lesssim,\gtrsim,\simeq$ in order to indicate dependence on other parameters. If the subscript is a function (e.g.~$\lesssim_f$), then it is understood as dependence on an upper bound on a suitable (weighted) Sobolev norm.
\item The ball centered at $x$ of radius $r$ in $\R^d$ is denoted by $B_r(x)$, and we simply write $B(x)=B_1(x)$, $B_r=B_r(0)$, and $B=B_1(0)$.
\item For a function $f$ and $1\le p<\infty$, we write $[f]_p(x):=(\fint_{B(x)}|f|^p)^{1/p}$ for the local $\Ld^p$ average, and similarly $[f]_\infty(x):=\sup_{B(x)}|f|$.
\item We systematically use {Einstein's summation convention} on repeated indices.
\item  For $a,b\in\R$, we write $a\vee b:=\max\{a,b\}$ and $a\wedge b:=\min\{a,b\}$.
\end{enumerate}
\end{samepage}

\section{Preliminary}

We first review useful results from Malliavin calculus for the fine analysis of functionals of the underlying Gaussian field $G$. Next, we recall several tools from quantitative stochastic homogenization theory, including optimal corrector estimates and annealed Calder\'on-Zygmund theory for linear elliptic equations with random coefficients.

\subsection{Malliavin calculus}

Since the covariance function $c$ is uniformly bounded (cf.~\eqref{eq:cov-as}), the Gaussian random field $G$ can be viewed as a random Schwartz distribution, that is, as a random element in $\Sc'(\R^d)^\kappa$: for all $\zeta_1,\zeta_2\in C^\infty_c(\R^d)^\kappa$ we define $G(\zeta_1)$, $G(\zeta_2)$ (or $\int_{\R^d}G\zeta_1$, $\int_{\R^d}G\zeta_2$) as centered Gaussian random variables with covariance
\[\cov{G(\zeta_1)}{G(\zeta_2)}\,:=\,\iint_{\R^d\times\R^d} \zeta_1(x)\cdot c(x-y)\,\zeta_2(y)\,dxdy.\]
We define  $\Hf$ as the closure of $C^\infty_c(\R^d)^\kappa$ for the (semi)norm
\[\|\zeta_1\|_{\Hf}^2:=\langle\zeta_1,\zeta_1\rangle_\Hf,\qquad\langle\zeta_1,\zeta_2\rangle_\Hf:=\iint_{\R^d\times\R^d} \zeta_1(x)\cdot c(x-y)\,\zeta_2(y)\,dxdy.\]
The space $\Hf$ (up to taking the quotient with respect to the kernel of $\|\cdot\|_\Hf$) is a separable Hilbert space.
In view of the isometry relation $\cov{G(\zeta_1)}{G(\zeta_2)}=\langle\zeta_1,\zeta_2\rangle_\Hf$, the random field $G$ is said to be an {isonormal Gaussian process} over $\Hf$.

We recall some basic definitions of the Malliavin calculus with respect to the Gaussian field $G$ (see e.g.~\cite{Malliavin-97,Nualart,NP-book} for details).
Without loss of generality, we work under the minimality assumption $\F=\sigma(G)$,
which implies that the linear subspace
\[\Rc:=\Big\{g\big(G(\zeta_1),\ldots,G(\zeta_n)\big)\,:\,n\in\N,\,g\in C_c^\infty(\R^n),\,\zeta_1,\ldots,\zeta_n\in C_c^\infty(\R^d)^\kappa\Big\}\subset\Ld^2(\Omega)\]
is dense in $\Ld^2(\Omega)$. This allows to define operators and prove properties on the simpler subspace $\Rc$ before extending them to $\Ld^2(\Omega)$ by density.
For $r\ge1$ we similarly define
\[\Rc(\Hf^{\otimes r}):=\Big\{\sum_{i=1}^n\psi_iX_i\,:\,n\in\N,\,X_1,\ldots, X_n\in\Rc,\,\psi_1,\ldots,\psi_n\in \Hf^{\otimes r}\Big\}\subset\Ld^2(\Omega;\Hf^{\otimes r}),\]
which is dense in $\Ld^2(\Omega;\Hf^{\otimes r})$.
For a random variable $X\in \Rc$, say $X=g(G(\zeta_1),\ldots,G(\zeta_n))$, we define its {Malliavin derivative} $DX\in \Ld^2(\Omega;\Hf)$ as
\begin{align}\label{eq:D-expl}
DX\,=\,\sum_{i=1}^n\zeta_i \,\partial_i g (G(\zeta_1),\ldots,G(\zeta_n)).
\end{align}
For an element $X\in \Rc(\Hf^{\otimes r})$ with $r\ge1$, say $X=\sum_{i=1}^n\psi_iX_i$, the Malliavin derivative $DX\in\Ld^2(\Omega;\Hf^{\otimes(r+1)})$ is then given by
$DX=\sum_{i=1}^n\psi_i\otimes DX_i$.
For $j\ge1$, we iteratively define the $j$th-order Malliavin derivative $D^j:\Rc(\Hf^{\otimes r})\to\Ld^2(\Omega;\Hf^{\otimes (r+j)})$ for all $r\ge0$.
For all $r,m\ge0$, we then set
\begin{gather*}
\langle X,Y\rangle_{\Dm^{m,2}(\Hf^{\otimes r})}:=\expec{\langle X,Y\rangle_{\Hf^{\otimes r}}}+\sum_{j=1}^m\expec{\langle D^jX,D^jY\rangle_{\Hf^{\otimes(r+j)}}},
\end{gather*}
we define the {Malliavin-Sobolev space} $\Dm^{m,2}(\Hf^{\otimes r})$ as the closure of $\Rc(\Hf^{\otimes r})$ for the corresponding norm, and we extend the Malliavin derivatives $D^j$ by density to these spaces.
Next, we define a {divergence operator} $D^*$ as the adjoint of the Malliavin derivative $D$,
and we construct the so-called {Ornstein-Uhlenbeck operator}
\[\Lc:=D^* D,\]
which is an essentially self-adjoint nonnegative operator.
We refer e.g.\@ to~\cite[p.34]{NP-book} for a description of the explicit action of $D^*$ and $\Lc$ on $\Rc$.
 In particular, it is easily checked that $\Lc$ commutes with shifts.
In addition, a direct computation (e.g.~\cite[p.35]{NP-book}) leads to the commutator relation
\begin{align}\label{eq:commut-DL}
D\Lc=(1+\Lc)D.
\end{align}

Based on the above definitions, we state the following proposition, which collects various useful results for the fine analysis of functionals of the Gaussian field $G$.
Item~(i) is classical.
Item~(ii) is best known in the discrete Gaussian setting~\cite{HS-94}.
Item~(iii) in total variation distance is a consequence of Stein's method: it was first obtained in the discrete setting by Chatterjee~\cite{C2}, while the present Malliavin analogue is due to~\cite{NP-08,NPR-09}. The corresponding result in $2$-Wasserstein distance is of a different nature and is due to~\cite{LNP-15}.
A proof and precise references are included in~\cite[Appendix~A]{DO1}.
Note that since $\Lc$ is nonnegative the inverse operator $(1+\Lc)^{-1}$ is well-defined and has operator norm bounded by $1$.

\begin{prop}[\cite{HS-94,C2,NP-08,NPR-09,LNP-15}]\label{prop:Mall}$ $
\begin{enumerate}[(i)]
\item \emph{First-order Poincaré inequality:} For all $X\in \Ld^2(\Omega)$,
\[\var{X}\le\expec{\|DX\|_{\Hf}^2}.\]
\item \emph{Helffer-Sjöstrand identity:} For all $X,Y\in \Dm^{1,2}(\Omega)$,
\begin{equation}\label{eq:HS}
\cov{X}{Y}=\expec{\langle DX,(1+\Lc)^{-1}DY\rangle_{\Hf}}.
\end{equation}
\item \emph{Second-order Poincaré inequality:} For all $X\in \Ld^2(\Omega)$ with $\expec{X}=0$ and $\var{X}=1$,
\begin{eqnarray*}
\quad\dWW{X}{\Nc}\vee\dTV{X}{\Nc}&\le& 2\,\var{\langle DX,(1+\Lc)^{-1}DX\rangle_{\Hf}}^\frac12\\
&\le&3\,\expec{\|D^2X\|_{\op}^4}^\frac14\expec{\|DX\|_\Hf^4}^\frac14,
\end{eqnarray*}
where $\dWW\cdot\Nc$ and $\dTV\cdot\Nc$ denote the $2$-Wasserstein and the total variation distances to a standard Gaussian law, respectively, and where the operator norm of $D^2X$ is defined by
\begin{equation}\label{eq:def-op}
\|D^2X\|_{\op}\,:=\,\sup_{\zeta,\zeta'\in\Hf\atop\|\zeta\|_\Hf=\|\zeta'\|_\Hf=1}\langle D^2X,\zeta\otimes\zeta'\rangle_{\Hf^{\otimes2}}.
\qedhere
\end{equation}
\end{enumerate}
\end{prop}

For later purposes, it is useful to transform the norm of $\Hf$ into a suitable Lebesgue norm. This is a variant of the Hardy-Littlewood-Sobolev inequality.

\begin{lem}[Hardy-Littlewood-Sobolev inequality]\label{lem:est-Hf}
For all $h\in C^\infty_c(\R^d)^\kappa$,
\[\|h\|_\Hf\,\lesssim\,
\left\{
\begin{array}{lll}
\|h\|_{\Ld^2(\R^d)}&:&\beta>d,\\
\|\log(2+|\cdot|)^\frac12h\|_{\Ld^{2}(\R^d)}&:&\beta=d,\\
\|h\|_{\Ld^\frac{2d}{2d-\beta}(\R^d)}&:&\beta<d.
\end{array}\right.
\qedhere\]
\end{lem}

\begin{proof}
For $\beta<d$, the estimate is a direct consequence of the Hardy-Littlewood-Sobolev inequality. For $\beta>d$, the inequality $2ab\le a^2+b^2$ implies
\[\int_{\R^d}\int_{\R^d}\frac{|h(x)||h(y)|}{(1+|x-y|)^\beta}\,dxdy\,\le\,\int_{\R^d}\int_{\R^d}\frac{|h(x)|^2}{(1+|x-y|)^\beta}\,dxdy\,\simeq\,\|h\|_{\Ld^2(\R^d)}^2.\]
We turn to the critical case $\beta=d$. Smuggling in the weight $\log(2+|x|)^\frac12$ and using Cauchy-Schwarz' inequality,
\begin{multline*}
\int_{\R^d}\int_{\R^d}\frac{|h(x)||h(y)|}{(1+|x-y|)^d}\,dxdy\\
\,\le\,\|\log(2+|\cdot|)^\frac12h\|_{\Ld^2(\R^d)}\bigg(\int_{\R^d}\log(2+|x|)^{-1}\Big(\int_{\R^d}\frac{|h(y)|}{(1+|x-y|)^d}dy\Big)^2dx\bigg)^\frac12.
\end{multline*}
Smuggling in the weight $(1+|y|)^\frac12$ and using Cauchy-Schwarz' inequality again,
\begin{multline*}
\int_{\R^d}\int_{\R^d}\frac{|h(x)||h(y)|}{(1+|x-y|)^d}\,dxdy
\,\le\,\|\log(2+|\cdot|)^\frac12h\|_{\Ld^2(\R^d)}\\
\times\bigg(\int_{\R^d}\log(2+|x|)^{-1}\Big(\int_{\R^d}\frac{(1+|y|)|h(y)|^2}{(1+|x-y|)^d}dy\Big)\Big(\int_{\R^d}\frac{dy}{(1+|y|)(1+|x-y|)^{d}}\Big)dx\bigg)^\frac12.
\end{multline*}
The last integral in brackets is controlled by $C\frac{\log(2+|x|)}{1+|x|}$, so that by Fubini's theorem,
\begin{multline*}
\int_{\R^d}\int_{\R^d}\frac{|h(x)||h(y)|}{(1+|x-y|)^d}\,dxdy
\,\lesssim\,\|\log(2+|\cdot|)^\frac12h\|_{\Ld^2(\R^d)}\\
\times\bigg(\int_{\R^d}(1+|y|)|h(y)|^2\Big(\int_{\R^d}\frac{dx}{(1+|x|)(1+|x-y|)^d}\Big)dy\bigg)^\frac12.
\end{multline*}
Using again that the last integral in brackets is controlled by $C\frac{\log(2+|y|)}{1+|y|}$, the conclusion follows.
\end{proof}

\subsection{Tools from quantitative stochastic homogenization}

Next to the corrector $\phi$, we recall the notion of the flux corrector $\sigma$.
The pair $(\phi,\sigma)$ is only defined up to an additive (random) constant and we choose the standard anchoring $\fint_B(\phi,\sigma)=0$ on the unit ball~$B$ at the origin.

\begin{lem}[Correctors, e.g.~\cite{GNO-reg}]\label{si}
Let the coefficient field $\Aa$ be stationary and ergodic (as is the case if $\Aa$ is Gaussian with parameter $\beta>0$). Then there exist two random tensor fields
$(\phi_i)_{1\le i\le d}$ and $(\sigma_{ijk})_{1\le i,j,k\le d}$ with the following properties:
\begin{enumerate}[$\bullet$]
\item The gradient fields $\nabla\phi_i$ and $\nabla\sigma_{ijk}$ are stationary\footnote{That is, shift-covariant: $\nabla\phi_i(\cdot+z;\Aa)=\nabla\phi_i(\cdot;\Aa(\cdot+z))$ and $\nabla\sigma_{ijk}(\cdot+z;\Aa)=\nabla\sigma_{ijk}(\cdot;\Aa(\cdot+z))$ almost everywhere in~$\R^d$, for all shift vectors $z\in\mathbb{R}^d$.}
and have finite second moments and vanishing expectations.
\item For all $i$ the matrix field $\sigma_i=(\sigma_{ijk})_{1\le j,k\le d}$ is skew-symmetric (that is, $\sigma_{ijk}=-\sigma_{ikj}$).
\item The following equations are satisfied a.s.\@ in the distributional sense on $\R^d$,
\begin{equation}\label{eq:corr}
\quad
-\nabla\cdot \Aa(\nabla\phi_i+e_i)\,=\,0,\qquad
\nabla\cdot\sigma_i\,=\,q_i,\qquad
-\triangle\sigma_{ijk}\,=\,\partial_jq_{ik}-\partial_kq_{ij},
\end{equation}
where $q_i=(q_{ij})_{1\le j\le d}$ denotes the centered flux,
\[\quad q_i:=\Aa(\nabla\phi_i+e_i)-\bar\Aa e_i,\qquad\bar\Aa\ee_i:=\expec{\Aa(\nabla\phi_i+\ee_i)}.\]
\end{enumerate}
In addition Meyers's higher-integrability result holds in the following form: there exists $\delta\simeq1$ such that $\expec{|(\nabla\phi,\nabla\sigma)|^{2(1+\delta)}}\lesssim1$.
\end{lem}

We recall the moment bounds satisfied by correctors in the present Gaussian setting.
For the corrector gradients, the stochastic integrability (i.e.\@ dependence on~$p$) is optimal.

\begin{lem}[Corrector estimates, \cite{AKM2,GNO-reg,GNO-quant}]\label{lem:cor}
Let the coefficient field $\Aa$ be Gaussian with parameter $\beta>0$ and let $\mu_{d,\beta}$ be as
in~\eqref{mudbeta}.
Then, the extended corrector $(\phi,\sigma)$ satisfies for all $1\le p<\infty$,
\[\expecm{[(\nabla \phi,\nabla \sigma)]_2^p}^\frac1p \,\lesssim \,
\left\{\begin{array}{lll}
p^\frac12&:&\beta>d,\\
(p\log p)^\frac12&:&\beta=d,\\
p^\frac{d}{2\beta}&:&\beta<d,
\end{array}\right.\]
and for all $x\in \R^d$,
\[\expecm{[(\phi,\sigma)]_2(x)^p}^\frac1p\,\lesssim_{p}\,
\mu_{d,\beta}(|x|).
\qedhere\]
\end{lem}

Finally, we state a useful annealed Calder\'on-Zygmund estimate for the elliptic equation with random coefficients. This result is due to~\cite[Section~6]{DO1} and constitutes a useful upgrade of the quenched Calder\'on-Zygmund estimates of~\cite{Armstrong-Daniel-16,AKM-book,GNO-reg}.
To obtain the weighted estimates as stated below, it suffices to use~\cite[Theorem~3.4]{Shen-07} instead of~\cite[Theorem~3.2]{Shen-07} in the proof of~\cite{DO1}.
\begin{prop}[Annealed Calder\'on-Zygmund estimate, \cite{DO1}]\label{prop:ann}
For $h\in C^\infty_c(\R^d;\Ld^\infty(\Omega))^d$, the unique Lax-Milgram solution of
\[-\nabla\cdot\Aa\nabla z=\nabla\cdot h\]
satisfies for all $1<p,q<\infty$, all weights $w$ in the Muckenhoupt class $A_p$, and all $0<\delta\le\frac12$,
\begin{multline*}
\|[\nabla z]_2\|_{\Ld^p_w(\R^d;\Ld^{q}(\Omega))}\,\lesssim_{p,q,w}\,
\|[h]_2\|_{\Ld^p_w(\R^d;\Ld^{q+\delta}(\Omega))}\\
\times|\!\log\delta|^{2|\frac1{q}-\frac1p|}
\left\{\begin{array}{lll}
(\tfrac1\delta)^{\frac1{p\wedge q\wedge2}-\frac1{p\vee q\vee2}}&:&\beta>d,\\
(\tfrac1\delta|\!\log\delta|)^{\frac1{p\wedge q\wedge2}-\frac1{p\vee q\vee2}}&:&\beta=d,\\
(\tfrac1\delta)^{\frac d\beta(\frac1{p\wedge q\wedge2}-\frac1{p\vee q\vee2})}&:&\beta<d.
\end{array}\right.
\end{multline*}
In particular, in the regime $|\!\log\delta|(|\frac1p-\frac12|+|\frac1q-\frac12|)\lesssim1$, the constant in this estimate can be chosen independent of $\delta$.
\end{prop}

\section{Convergence of the covariance structure}\label{chap:conv-cov}

In this section, we establish the convergence of the covariance structure for the rescaled homogenization commutator, thus proving Theorem~\ref{th:main}(i). More precisely, we establish the following result.
Note that in the non-integrable case oscillations in the covariance structure of $G$ can break down the convergence. Likewise, convergence rates can be arbitrarily slow.

\begin{prop}\label{prop:conv-cov}
Let $\Aa=a_0(G)$ be Gaussian with parameter $\beta>0$.
For $1\le l,m\le \kappa$, define the matrix $\Kc^l$ by
\[\Kc_{ij}^{l}:=\expec{(\nabla\phi_j^*+\ee_j)\cdot \partial_la_0(G)(\nabla\phi_i+\ee_i)},\]
and define the measurable tensor field $K^{lm}$ of order $4$ on $\R^d$ by
\begin{multline*}
K_{iji'j'}^{lm}(x):=\E\Big[\big((\nabla\phi_j^*+\ee_j)\cdot \partial_la_0(G)(\nabla\phi_i+\ee_i)\big)(x)\\
(1+\Lc)^{-1}\big((\nabla\phi_{j'}^*+\ee_{j'})\cdot\partial_ma_0(G)(\nabla\phi_{i'}+\ee_{i'})\big)(0)\Big],
\end{multline*}
which satisfies $\|[K]_1\|_{\Ld^\infty(\R^d)}\lesssim1$.
\begin{enumerate}[(i)]
\item \emph{Integrable case $\beta>d$:}
For all $F,F'\in C^\infty_c(\R^d)^{d\times d}$,
\begin{multline*}
\qquad\Big|\cov{I_\e(F)}{I_\e(F')}-\int_{\R^d} F(x):\calQ:F'(x)\,dx\Big|\\
\,\lesssim_{F,F'}\,\left\{\begin{array}{lll}
\e&:&d>2,\,\beta\ge d+1,\\
\e\Log^\frac12&:&d=2,\,\beta\ge d+1,\\
\e^{\beta-d}&:&d<\beta<d+1,
\end{array}\right.
\end{multline*}
where the effective fluctuation tensor $\calQ$ is given by
\[\calQ_{iji'j'}:=\int_{\R^d} K_{iji'j'}^{lm}(x)\,c_{lm}(x)\,dx.\]
\item \emph{Critical case $\beta=d$:}
For all $F,F'\in C^\infty_c(\R^d)^{d\times d}$,
\begin{multline*}
\qquad\bigg|\cov{I_\e(F)}{I_\e(F')}
-\Big(\int_{\R^d}F(x):\Kc^l\otimes\Kc^m:F'(x)\,dx\Big)\,\Big(\frac1\Log\int_{|y|<\frac1\e}c_{lm}(y)\,dy\Big) \bigg|\\
\,\lesssim_{F,F'}\,\Log^{-1}.
\end{multline*}
In particular, the limit $\lim_{\e\downarrow0}\cov{I_\e(F)}{I_\e(F')}$ exists for all $F,F'$ if and only if the limit
\[\bar C_{lm}\,:=\,\lim_{L\uparrow\infty}\frac1{\log L}\int_{B_L}c_{lm}(y)\,dy\]
exists for all $l,m$ with $\Kc^l\ne0\ne\Kc^m$. In that case,
\[\lim_{\e\downarrow0}\,\cov{I_\e(F)}{I_\e(F')}=\int_{\R^d} F(x):\calQ:F'(x)\,dx,\]
where the effective fluctuation tensor $\calQ$ is given by
\[\calQ_{iji'j'}\,:=\,\Kc_{ij}^l\Kc^m_{i'j'}\bar C_{lm}.\]
\item \emph{Non-integrable case $\beta<d$:}
For all $F,F'\in C^\infty_c(\R^d)$,
\begin{multline*}
\qquad\bigg|\cov{I_\e(F)}{I_\e(F')}
-\int_{\R^d}\int_{\R^d}F(x):\Kc^l\otimes\Kc^m:F'(y)\,\tfrac1{\e^{\beta}}c_{lm}\big(\tfrac1\e (x-y)\big)\,dxdy\bigg|\\
\,\lesssim_{F,F'}\,\e\mu_{d,\beta}(\tfrac1\e)+ 
 \left\{\begin{array}{lll}
\e^{d-\beta}&:&\beta>\frac d2,\\
\e^\frac d2\Log&:&\beta=\frac d2,\\
\e^\beta&:&\beta<\frac d2.
\end{array}\right.
\end{multline*}
In particular, the limit $\lim_{\e\downarrow0}\cov{I_0^\e(F)}{I_0^\e(F')}$ exists for all $F,F'$ if and only if the function $L^{\beta}c_{lm}(L\cdot)$ converges weakly-* in $\Ld^\infty(\Sp^{d-1})$ to some function $C_{lm}$ as $L\uparrow\infty$ for all $l,m$ with $\Kc^l\ne0\ne\Kc^m$. In that case,
\[\lim_{\e\downarrow0}\,\cov{I_\e(F)}{I_\e(F')}\,=\,\int_{\R^d}\int_{\R^d} F(x):\frac{\calQ(\frac{x-y}{|x-y|})}{|x-y|^\beta}:F'(y)\,dxdy,\]
where the effective fluctuation tensor field $\calQ$ on $\Sp^{d-1}$ is given by
\[\calQ_{iji'j'}(u)\,:=\,\Kc_{ij}^l\Kc^m_{i'j'}C_{lm}(u).\qedhere\]
\end{enumerate}
\end{prop}

\begin{proof}
By polarization,
it is enough to consider the case $F=F'=g\,\ee_i\otimes \ee_j$ for all $i,j$ and $g\in C^\infty_c(\R^d)$.
We aim at analyzing the limit of the variance
\[\nu_\e(g)\,:=\,\var{\pi_{d,\beta}(\tfrac1\e)^\frac12\int_{\R^d} g(x)\,\Xi_{ij}(\tfrac x\e)\,dx}\,=\,\var{\int_{\R^d} g_\e\,\Xi_{ij}},\]
where we have set $g_\e(x):=\e^{d}\pi_{d,\beta}(\tfrac1\e)^\frac12g(\e x)$.
We split the proof into five steps.

\medskip
\step{1} Representation formula for the Malliavin derivative of the homogenization commutator,
\begin{multline}\label{eq:form-D-Xi}
D\int_{\R^d} g_\e\Xi_{ij}=g_\e\,(\nabla\phi_j^*+\ee_j)\cdot \partial a_0(G) (\nabla\phi_i+\ee_i)\\
+(\nabla z_{\e,j}+\phi_j^*\nabla g_\e)\cdot\partial a_0(G)(\nabla\phi_i+\ee_i),
\end{multline}
where the auxiliary field $z_{\e,j}$ is the unique Lax-Milgram solution in $\R^d$ of
\begin{equation}\label{eq:def-zeps}
-\nabla\cdot\Aa^*\nabla z_{\e,j}=\nabla\cdot\big((\Aa^*\phi_j^*-\sigma_j^*)\nabla g_\e\big).
\end{equation}
Indeed, by definition of the homogenization commutator,
\begin{eqnarray*}
D\Xi_{ij}
&=&\ee_j\cdot D\Aa\,(\nabla\phi_i+\ee_i)+\ee_j\cdot(\Aa-\bar\Aa)\nabla D\phi_i.
\end{eqnarray*}
Using the definition of the flux corrector $\sigma_j^*$ in the form $(\Aa^*-\bar\Aa^*)\ee_j=-\Aa^*\nabla\phi_j^*+\nabla\cdot\sigma_j^*$ and using the skew-symmetry of $\sigma_j^*$,
we find
\begin{eqnarray*}
D\Xi_{ij}&=&\ee_j\cdot D\Aa\,(\nabla\phi_i+\ee_i)+(\nabla\cdot\sigma_j^*)\cdot\nabla D\phi_i-\nabla\phi_j^*\cdot\Aa\nabla D\phi_i\\
&=&\ee_j\cdot D\Aa\,(\nabla\phi_i+\ee_i)-\nabla\cdot\big((\Aa\phi_j^*+\sigma_j^*)\nabla D\phi_i\big)+\phi_j^*\nabla\cdot\Aa\nabla D\phi_i.
\end{eqnarray*}
Using the corrector equation~\eqref{eq:corr} for $\phi_i$ in the form
\begin{equation}\label{eq:Dphii}
-\nabla\cdot\Aa\nabla D\phi_i=\nabla\cdot D\Aa(\nabla\phi_i+\ee_i),
\end{equation}
we deduce
\begin{equation*}
D\Xi_{ij}\,=\,(\nabla\phi_j^*+\ee_j)\cdot D\Aa(\nabla\phi_i+\ee_i)
-\nabla\cdot\big((\Aa\phi_j^*+\sigma_j^*)\nabla D\phi_i\big)-\nabla\cdot\big( \phi_j^*D\Aa(\nabla \phi_i+\ee_i)\big).
\end{equation*}
Integrating with the test function $g_\e$ yields
\begin{multline*}
D\int_{\R^d}g_\e\Xi_{ij}=\int_{\R^d}g_\e\,(\nabla\phi_j^*+\ee_j)\cdot D\Aa(\nabla\phi_i+\ee_i)+\int_{\R^d}\nabla g_\e\cdot(\Aa\phi_j^*+ \sigma_j^*)\nabla D\phi_i\\
+\int_{\R^d}\phi_j^*\nabla g_\e\cdot D\Aa(\nabla\phi_i+\ee_i).
\end{multline*}
Using the equation for $z_{\e,j}$, the skew-symmetry of $\sigma_j^*$, and  the corrector equation for $\phi_i$ in the form~\eqref{eq:Dphii}, we may reformulate the second right-hand side term as
\[\int_{\R^d} \nabla g_\e\cdot(\Aa\phi_j^*+ \sigma_j^*)\nabla D\phi_i=-\int_{\R^d} \nabla z_{\e,j}\cdot\Aa \nabla D\phi_i=\int_{\R^d} \nabla z_{\e,j}\cdot D\Aa(\nabla\phi_i+e_i).\]
Further noting that
\begin{equation}\label{eq:Da}
D_z\Aa=\partial a_0(G(z))\,\delta(\cdot-z),
\end{equation}
the claim~\eqref{eq:form-D-Xi} follows (since $a_0$ is Lipschitz).

\medskip
\step2 Application of the Helffer-Sj\"ostrand identity.

\noindent
By Proposition~\ref{prop:Mall}(ii), we may  represent the variance $\nu_\e(g)$ as
\begin{equation*}
\nu_\e(g)\,=\,\expec{\bigg\langle D\Big(\int_{\R^d} g_\e\Xi_{ij}\Big)\,,\,(1+\Lc)^{-1}D\Big(\int_{\R^d} g_\e\Xi_{ij}\Big)\bigg\rangle_{\Hf}\,}.
\end{equation*}
By~\eqref{eq:form-D-Xi}, the boundedness of $(1+\Lc)^{-1}$ on $\Ld^2(\Omega;\Hf)$, and the stationarity of $(\nabla\phi_j^*+\ee_j)\cdot \partial_la_0(G) (\nabla\phi_i+\ee_i)$, recalling that~$\Lc$ commutes with shifts, this leads to
\begin{equation}\label{eq:decomp-rep-diff-nu12+cov}
|\nu_\e(g)-U_\e|\le 2(S_\e T_\e)^\frac12+T_\e,
\end{equation}
in terms of
\begin{eqnarray}
U_\e&:=&\int_{\R^d} \int_{\R^d} g_\e(x)g_\e(y)\,K_{ijij}^{lm}(x-y)\,c_{lm}(x-y)\,dxdy,\label{e.defUeps}\\
S_\e&:=&\expec{\big\|g_\e(\nabla\phi_j^*+\ee_j)\cdot\partial a_0(G)(\nabla\phi_i+\ee_i)\big\|_{\Hf}^2},\label{e.defSeps} \\
T_\e&:=&\expec{\big\|(\nabla z_{\e,j}+\phi_j^*\nabla g_\e)\cdot\partial a_0(G)(\nabla\phi_i+\ee_i)\big\|_{\Hf}^2}, \label{e.defTeps}
\end{eqnarray}
where $K$ is the tensor field defined in the statement of the proposition and where we recall that $z_{\e,j}$ is defined in~\eqref{eq:def-zeps}.

\medskip
\step3 Properties of $K$: we show that
\begin{eqnarray}
\|[K]_1\|_{\Ld^\infty}&\lesssim&1,\label{eq:K-bound}\\
|(K_{ijij}^{lm})_1(x)-\Kc_{ij}^l\Kc_{ij}^m|&\lesssim&(1+|x|)^{-\beta}\times\left\{\begin{array}{lll}1
&:&\beta< d,\\
\log^2(2+|x|)&:&\beta=d,\end{array}\right.\label{eq:K-conv}
\end{eqnarray}
where for a measurable function $G$ on $\R^d$ we use the following short-hand notation for the local average,
\[(G)_1(x):=\fint_{B} \fint_{B} G(x+y+y')\,dydy'.\]
We start with~\eqref{eq:K-bound}: by stationarity, the boundedness of $(1+\Lc)^{-1}$ on $\Ld^2(\Omega)$, and the corrector estimates of Lemma~\ref{lem:cor}, recalling that~$\Lc$ commutes with shifts and that $a_0$ is Lipschitz, we find
\begin{eqnarray*}
[K_{ijij}^{lm}]_1(x)&\lesssim& \fint_{B_2} \fint_{B_2} |K_{ijij}^{lm}(x+y+y')| dydy'
\\
&\le& \mathbb E \bigg[\Big(\fint_{B_2(x)}\big|(\nabla\phi_j^*+\ee_j)\cdot \partial_la_0(G)(\nabla\phi_i+\ee_i)\big|\Big)
\\
&&\qquad\times\Big(\fint_{B_2}\big|(1+\Lc)^{-1}\big((\nabla\phi_{j}^*+\ee_{j})\cdot\partial_ma_0(G)(\nabla\phi_{i}+\ee_{i})\big)\big|\Big)\bigg]
\\
&\lesssim &\expec{[\nabla\phi^*+\Id]_2^4}^\frac12\expec{[\nabla\phi+\Id]_2^4}^\frac12\,\lesssim \, 1.
\end{eqnarray*}
We turn to~\eqref{eq:K-conv}.
Since the Gaussian field $G$ is strongly mixing (as the covariance function decays at infinity), and  since the identity $\Lc1=0$ and the essential self-adjointness of~$\Lc$ ensure $\expec{(1+\Lc)^{-1}u}= \expec{u}$ for all $u\in\Ld^2(\Omega)$, it directly follows from the stationarity of $(\nabla\phi_j^*+\ee_j)\cdot \partial_la_0(G)(\nabla\phi_i+\ee_i)$ that
\begin{eqnarray*}
\lefteqn{\lim_{|x|\uparrow\infty} (K_{ijij}^{lm})_1(x)}\nonumber\\
&=&\expec{(\nabla\phi_j^*+\ee_j)\cdot \partial_la_0(G)(\nabla\phi_i+\ee_i)}\expec{(1+\Lc)^{-1}\big((\nabla\phi_j^*+\ee_j)\cdot\partial_ma_0(G)(\nabla\phi_i+\ee_i)\big)} 
\\
&=&\expec{(\nabla\phi_j^*+\ee_j)\cdot \partial_la_0(G)(\nabla\phi_i+\ee_i)}\expec{(\nabla\phi_j^*+\ee_j)\cdot \partial_ma_0(G)(\nabla\phi_i+\ee_i)}
\\
&=&\Kc_{ij}^{l}\Kc_{ij}^{m},
\end{eqnarray*}
and it remains to establish a convergence rate.
Starting from
\begin{multline*}
(K^{lm}_{ijij})_1(x)-\Kc_{ij}^l\Kc_{ij}^m\,=\,\Cov\bigg[\Big(\fint_{B(x)}(\nabla\phi_j^*+\ee_j)\cdot\partial_la_0(G)(\nabla\phi_i+\ee_i)\Big);\\
(1+\Lc)^{-1}\Big(\fint_{B}(\nabla\phi_j^*+\ee_j)\cdot\partial_ma_0(G)(\nabla\phi_i+\ee_i)\Big)\bigg],
\end{multline*}
the Helffer-Sjöstrand identity of Proposition~\ref{prop:Mall}(ii) together with the commutation relation~\eqref{eq:commut-DL} leads to
\begin{multline}\label{eq:K-diff-cov}
(K^{lm}_{ijij})_1(x)-\Kc_{ij}^l\Kc_{ij}^m
=\,\E\bigg[\Big\langle D\Big(\fint_{B(x)}(\nabla\phi_j^*+\ee_j)\cdot\partial_la_0(G)(\nabla\phi_i+\ee_i)\Big),\\
(1+\Lc)^{-1}(2+\Lc)^{-1}D\Big(\fint_{B}(\nabla\phi_j^*+\ee_j)\cdot\partial_ma_0(G)(\nabla\phi_i+\ee_i)\Big)\Big\rangle_\Hf\bigg].
\end{multline}
(Note indeed that~\eqref{eq:commut-DL} yields $(1+\Lc)^{-1}D=(2+\Lc)^{-1}D$.)
Since $D_z\partial a_0(G)=\partial^2a_0(G(z))\,\delta(\cdot-z)$, the Malliavin derivative of the factors is evaluated as follows,
\begin{multline*}
D_z\Big(\fint_{B(x)}(\nabla\phi_j^*+\ee_j)\cdot\partial a_0(G)(\nabla\phi_i+\ee_i)\Big)\\
\,=\,|B|^{-1}\mathds{1}_{z\in B(x)} \big((\nabla\phi_j^*+\ee_j)\cdot \partial^2a_0(G)(\nabla\phi_i+\ee_i)\big)(z)\\
+\fint_{B(x)}(\nabla\phi_j^*+\ee_j)\cdot\partial a_0(G)\nabla D_z\phi_i+\fint_{B(x)}\nabla D_z\phi_j^*\cdot\partial a_0(G)(\nabla\phi_i+\ee_i).
\end{multline*}
Convolving with $c_0$ and recalling the corrector equation for $\phi_i$ in the form~\eqref{eq:Dphii} together with~\eqref{eq:Da},
\begin{multline*}
\int_{\R^d} c_0(z-z')\,D_{z'}\Big(\fint_{B(x)}(\nabla\phi_j^*+\ee_j)\cdot\partial a_0(G)(\nabla\phi_i+\ee_i)\Big)\,dz'\\
\,=\,\fint_{B(x)}c_0(z-\cdot)\,(\nabla\phi_j^*+\ee_j)\cdot \partial^2a_0(G)(\nabla\phi_i+\ee_i)\\
+\fint_{B(x)}(\nabla\phi_j^*+\ee_j)\cdot\partial a_0(G)\nabla\Phi_{z,i}+\fint_{B(x)}\nabla\Phi_{z,j}^*\cdot\partial a_0(G)(\nabla\phi_i+\ee_i),
\end{multline*}
where $\Phi_{z,i}$ and $\Phi_{z,j}^*$ denote the unique Lax-Milgram solutions of
\begin{eqnarray}
-\nabla\cdot\Aa\nabla\Phi_{z,i}&=&\nabla\cdot\big(c_0(z-\cdot)\partial a_0(G)(\nabla\phi_i+\ee_i)\big),\label{eq:Phiz}\\
-\nabla\cdot\Aa^*\nabla\Phi_{z,j}^*&=&\nabla\cdot\big(c_0(z-\cdot)\partial a_0^*(G)(\nabla\phi_j^*+\ee_j)\big).\nonumber
\end{eqnarray}
Inserting this representation formula into the right-hand side of~\eqref{eq:K-diff-cov}, noting that the operator $(1+\Lc)^{-1}(2+\Lc)^{-1}$ is bounded in $\Ld^2(\Omega)$, 
and using the corrector estimates of Lemma~\ref{lem:cor}, we find
\begin{multline}\label{100.1}
|(K^{lm}_{ijij})_1(x)-\Kc_{ij}^l\Kc_{ij}^m|
\,\lesssim\,\int_{\R^d}\Big(|c_0(z-x)|+\expec{[\nabla\Phi_{z}]_2^4(x)+[\nabla\Phi_{z}^*]_2^4(x)}^\frac14\Big)\\
\times \Big(|c_0(z)|+\expec{[\nabla\Phi_{z}]_2^4(0)+[\nabla\Phi_{z}^*]_2^4(0)}^\frac14\Big)\,dz.
\end{multline}
We expand the product appearing in the right-hand side and only treat one of the terms, showing that
\begin{multline}
J_\beta(x):=\int_{\R^d}\expec{[\nabla\Phi_{z}]_2^4(x)}^\frac14\expec{[\nabla\Phi_{z}]_2^4(0)}^\frac14\,dz\\
\,\lesssim\,(1+|x|)^{-\beta}\times\left\{\begin{array}{lll}1
&:&\beta< d,\\
\log^2(2+|x|)&:&\beta=d,\end{array}\right.\label{eq:K-conv+}
\end{multline}
while the other terms are similar.
Noting that $\Phi_{z,i}(\cdot;\Aa)=\Phi_{0,i}(\cdot-z;\Aa(\cdot+z))$, we find
\[J_\beta(x)=\int_{\R^d}\expec{[\nabla\Phi_{0}]_2^4(x+z)}^\frac14\expec{[\nabla\Phi_{0}]_2^4(z)}^\frac14\,dz.\]
We start with the case $\beta<d$.
Smuggling in the weight $(1+|z|)^{\frac{\beta+d}{4}} (1+|x+z|)^{-\frac{\beta+d}{4}}$ and applying Cauchy-Schwarz' inequality,
\[J_\beta(x)\lesssim\int_{\R^d}\big(1+|z-x|\wedge|z+x|\big)^{-\frac{\beta+d}{2}} (1+|z|)^{\frac{\beta+d}{2}}\expec{[\nabla\Phi_{0}]_2^4(z)}^\frac12dz.\]
Since the weight $z \mapsto (1+|z-x|\wedge|z+x|)^{-\frac{\beta+d}2}(1+|z|)^{\frac{\beta+d}2}$ belongs to the Muckenhoupt class $A_2$, applying the weighted annealed Calder\'on-Zygmund estimate of Proposition~\ref{prop:ann} to equation~\eqref{eq:Phiz}, and using the corrector estimates of Lemma~\ref{lem:cor}, we find for $\beta<d$,
\begin{eqnarray*}
J_\beta(x)\,\lesssim\,\int_{\R^d}\big(1+|z-x|\wedge|z+x|\big)^{-\frac{\beta+d}{2}} (1+|z|)^{\frac{\beta+d}{2}}c_0(z)^2dz\,\lesssim\,(1+|x|)^{-\beta},
\end{eqnarray*}
that is,~\eqref{eq:K-conv+}.

\smallskip\noindent
Finally, we turn to the proof of~\eqref{eq:K-conv+} in the critical case $\beta=d$.
In order to obtain the optimal power of the logarithm, we rather use the Green's representation formula for $\nabla\Phi_0$  and appeal to annealed bounds on the Green's function~\cite{MaO,GMa,AKM-book,GNO-reg,BG-18} in the form
\[\expec{[\nabla_x\nabla_y G]_2^p(x,y)}^{\frac1p} \,\lesssim_p\,(1+|y-x|)^{-d},\]
for $1\le p<\infty$.
Together with the corrector estimates of Lemma~\ref{lem:cor} and with the decay assumption~\eqref{eq:cov-as-bis}, this leads to
\[\expec{[\nabla\Phi_0]_2^4(x)}^\frac14\,\lesssim\,\int_{\R^d}(1+|x-y|)^{-d}|c_0(y)|\,dy\,\lesssim\,\frac{\log^\frac12(2+|x|)}{(1+|x|)^{d}},\]
hence,
\begin{equation*}
J_d(x)\,\lesssim\,\int_{\R^d}\frac{\log^\frac12(2+|x+z|)}{(1+|x+z|)^{d}}\frac{\log^\frac12(2+|z|)}{(1+|z|)^{d}}\,dz
\,\lesssim\,\frac{\log^2(2+|x|)}{(1+|x|)^{d}},
\end{equation*}
that is,~\eqref{eq:K-conv+}.

\medskip

\step4 Limit of $U_\e$ (cf.~\eqref{e.defUeps}).
\nopagebreak

\noindent
We start with the integrable case $\beta>d$. By definition of $\pi_{d,\beta}$,  a change of variables yields
\begin{eqnarray*}
U_\e&=&\e^d\int_{\R^d} \int_{\R^d} g(\e x)g(\e y)K_{ijij}^{lm}(x-y)c_{lm}(x-y)\,dxdy\\
&=&\int_{\R^d} \int_{\R^d} g(x+\e y)g(x)K_{ijij}^{lm}(y)c_{lm}(y)\,dxdy.
\end{eqnarray*}
Since $[K]_1$ is bounded (cf.~\eqref{eq:K-bound}) and $\int_{\R^d} [c]_\infty\lesssim1$, we deduce by dominated convergence,
\begin{align*}
\lim_{\e\downarrow0}U_\e=\|g\|_{\Ld^2(\R^d)}^2\int_{\R^d} K_{ijij}^{lm}(y)c_{lm}(y)\,dy.
\end{align*}
More precisely, splitting $\|g\|_{\Ld^2(\R^d)}^2=\frac12\int_{\R^d}(|g(x)|^2+|g(x+\e y)|^2)\,dx$, we find
\begin{eqnarray*}
\lefteqn{\bigg|U_\e-\|g\|_{\Ld^2}^2\int_{\R^d} K_{ijij}^{lm}(y)c_{lm}(y)\,dy\bigg|}\\
&\le&\frac12\int_{\R^d} \int_{\R^d}|g(x+\e y)-g(x)|^2|K(y)||c(y)|\,dxdy\\
&\le& \frac12\int_{\R^d} \bigg(\int_{\R^d}\Big(\sup_{y'\in B(y)} |g(x+\e y')-g(x)|^2\Big)dx\bigg)\,[K]_1(y)[c]_\infty(y)\,dy \\
&\lesssim&\|g\|_{H^1(\R^d)}^2\int_{\R^d}(1\wedge|\e y|)^2[c]_\infty(y)\,dy\\
&\lesssim_g&
\e^{2\wedge(\beta-d)}\big(1+\Log\mathds1_{\beta=d+2}\big).
\end{eqnarray*}
We turn to the non-integrable case $\beta<d$. By definition of $\pi_{d,\beta}$, we find after rescaling,
\begin{eqnarray*}
U_\e&=&\e^{-\beta}\int_{\R^d}\int_{\R^d}g(x)g(y)\,K_{ijij}^{lm}\big(\tfrac1\e (x-y)\big)\,c_{lm}\big(\tfrac1\e (x-y)\big)dxdy\\
&=&\int_{\R^d}\int_{\R^d} \frac{g(x)g(y)}{|x-y|^\beta}K_{ijij}^{lm}\big(\tfrac1\e (x-y)\big)\,\big(\tfrac1\e|x-y|\big)^{\beta}c_{lm}\big(\tfrac1\e (x-y)\big)\,dxdy.
\end{eqnarray*}
Before applying~\eqref{eq:K-conv}, we take local averages and define
$$
\tilde U_\e\,:=\, \int_{\R^d}\int_{\R^d} \frac{g(x)g(y)}{|x-y|^\beta}(K_{ijij}^{lm})_1\big(\tfrac1\e (x-y)\big)\,\big(\tfrac1\e|x-y|\big)^{\beta}c_{lm}\big(\tfrac1\e (x-y)\big)\,dxdy,
$$
and we estimate the error
\begin{eqnarray*}
|\tilde U_\e-U_\e|&\lesssim&  \int_{\R^d}\int_{\R^d}\e^{1-\beta}\,[g]_\infty(x)[\nabla g]_\infty(y)\,[c]_\infty \big(\tfrac1\e (x-y)\big)\,dxdy
\\
&&+ \int_{\R^d}\int_{\R^d}\e^{-\beta} [g]_\infty(x)[g]_\infty(y) \,[\nabla c]_\infty \big(\tfrac1\e (x-y)\big)\,dxdy
\\
&\lesssim_g & \e,
\end{eqnarray*}
using the additional decay assumption $[\nabla c]_\infty(x)\lesssim (1+|x|)^{-\beta-1}$.
Next, we appeal to~\eqref{eq:K-conv} in the form
\begin{eqnarray*}
\lefteqn{\bigg|U_\e-\Kc_{ij}^l\Kc_{ij}^m\int_{\R^d} \int_{\R^d} \frac{g(x)g(y)}{|x-y|^\beta}\,\big(\tfrac1\e|x-y|\big)^{\beta}c_{lm}\big(\tfrac1\e (x-y)\big)\,dxdy\bigg|}\\
&\lesssim &|\tilde U_\e-U_\e|+{\bigg|\tilde U_\e-\Kc_{ij}^l\Kc_{ij}^m\int_{\R^d} \int_{\R^d} \frac{g(x)g(y)}{|x-y|^\beta}\,\big(\tfrac1\e|x-y|\big)^{\beta}c_{lm}\big(\tfrac1\e (x-y)\big)\,dxdy\bigg|}\\
&\lesssim_g&\e+\int_{\R^d}\int_{\R^d} \frac{|g(x)||g(y)|}{(\e+|x-y|)^{\beta}}\big|c\big(\tfrac1\e (x-y)\big)\big|\,dxdy\\
&\lesssim_g&\e+\e^\beta\int_{\R^d}\int_{\R^d} \frac{|g(x)||g(y)|}{(\e+|x-y|)^{2\beta}}\,dxdy\\
&\lesssim_g&\e+\e^{\beta\wedge(d-\beta)}\big(1+\Log\mathds1_{\beta=\frac d2}\big).
\end{eqnarray*}
It remains to analyze the critical case $\beta=d$. By definition of $\pi_{d,\beta}$, a change of variables yields
\begin{eqnarray*}
U_\e&=&\frac{\e^{d}}\Log\int_{\R^d} \int_{\R^d} g(\e x)g(\e y)K_{ijij}^{lm}(x-y)\,c_{lm}(x-y)\,dxdy\\
&=&\frac1\Log\int_{\R^d} \int_{\R^d} g(x+\e y)g(x)K_{ijij}^{lm}(y)c_{lm}(y)\,dxdy.
\end{eqnarray*}
Using the boundedness of $[K]_1$ (cf.~\eqref{eq:K-bound}) and the decay $|c(y)|\lesssim(1+|y|)^{-d}$, we find for~$p<2$,
\begin{multline*}
\lefteqn{\bigg|U_\e-\frac1\Log\int_{\R^d} \int_{|y|<\frac1\e} g(x+\e y)g(x)K_{ijij}^{lm}(y)c_{lm}(y)\,dxdy\bigg|}\\
\,\lesssim\,\frac{1}\Log\int_{\R^d} \int_{|y|>1} \frac{[g]_\infty(x+y)[g]_\infty(x)}{|y|^d}\,dxdy\,\lesssim_p\,\frac{\|[g]_\infty\|_{\Ld^p(\R^d)}^2}\Log,
\end{multline*}
hence,
\begin{eqnarray*}
\lefteqn{\bigg|U_\e-\frac{\|g\|_{\Ld^2}^2}\Log \int_{|y|<\frac1\e} K_{ijij}^{lm}(y)c_{lm}(y)\,dy\bigg|}\\
&\lesssim_p&\frac{\|[g]_\infty\|_{\Ld^p(\R^d)}^2}\Log+\frac1\Log\int_{\R^d} \int_{|y|<\frac1\e} \frac{\sup_{y'\in B(y)} |g(x+\e y')-g(x)|^2}{(1+|y|)^{d}}\,dxdy
\\
&\lesssim_{p,g}&\frac1\Log.
\end{eqnarray*}
Next, using~\eqref{eq:K-conv} as above, we conclude
\[\bigg|U_\e-\frac{\|g\|_{\Ld^2}^2\Kc_{ij}^l\Kc_{ij}^m}\Log\int_{|y|<\frac1\e}c_{lm}(y)\,dy\bigg|\,\lesssim_{p,g}\,\frac{1}\Log.\]

\medskip
\step4 Error estimates (cf.~\eqref{e.defSeps} and \eqref{e.defTeps}):
\begin{gather*}
S_\e\,\lesssim_g\,1\qquad\text{and}\qquad
T_\e\,\lesssim_g\,\e^2\mu_{d,\beta}(\tfrac1\e)^2.
\end{gather*}
We start with $S_\e$, and recall that 
\begin{align*}
S_\e:=~&\expec{\big\|g_\e(\nabla\phi_j^*+\ee_j)\cdot\partial a_0(G)(\nabla\phi_i+\ee_i)\big\|_{\Hf}^2}.
\end{align*}
By definition of the norm in $\Hf$, smuggling in local averages, we find
\begin{multline*}
\big\|g_\e(\nabla\phi_j^*+\ee_j)\cdot\partial a_0(G)(\nabla\phi_i+\ee_i)\big\|_{\Hf}^2
\,\lesssim\, \iint_{\R^d\times \R^d} \big([g_\e]_\infty [\nabla\phi^*+\Id]_2[\nabla\phi+\Id]_2\big)(x)\\
\times\big([g_\e]_\infty [\nabla\phi^*+\Id]_2[\nabla\phi+\Id]_2\big)(y)\,[c]_\infty(x-y)\,dxdy,
\end{multline*}
hence, by Lemma~\ref{lem:est-Hf} and the corrector estimates of~Lemma~\ref{lem:cor},
\begin{eqnarray*}
S_\e\,\lesssim\,
\left\{\begin{array}{lll}
\|[g_\e]_\infty \|_{\Ld^\frac{2d}{2d-\beta}(\R^d)}&:&\beta<d,\\
\|\log(2+|\cdot|)^\frac12[g_\e]_\infty\|_{\Ld^{2}(\R^d)}&:&\beta=d,\\
\|[g_\e]_\infty\|_{\Ld^2(\R^d)}&:&\beta>d,
\end{array}\right.
\end{eqnarray*}
and the claim $S_\e\lesssim_g1$ follows from the definition of $\pi_{d,\beta}$.
We turn to~$T_\e$, and recall that 
\begin{align*}
T_\e:=~&\expec{\big\|(\nabla z_{\e,j}+\phi_j^*\nabla g_\e)\cdot\partial a_0(G)(\nabla\phi_i+\ee_i)\big\|_{\Hf}^2}.
\end{align*}
In the integrable case $\beta>d$, Lemma~\ref{lem:est-Hf} and the corrector estimates of~Lemma~\ref{lem:cor}
similarly lead to
\begin{eqnarray*}
T_\e&\lesssim&\expec{\big\|[\nabla z_{\e}+\phi^*\nabla g_\e]_2[\nabla\phi+\Id]_2\big\|_{\Ld^2(\R^d)}^2}\\
&\lesssim&\|[\nabla z_{\e}]_2\|_{\Ld^2(\R^d;\Ld^4(\Omega))}^2+\|\mu_{d,\beta}[\nabla g_\e]_\infty\|_{\Ld^2(\R^d)}^2,
\end{eqnarray*}
while the annealed Calder\'on-Zygmund estimate of~Proposition~\ref{prop:ann} applied to equation~\eqref{eq:def-zeps} and combined with the corrector estimates of~Lemma~\ref{lem:cor} then implies
\[T_\e\,\lesssim\,\|\mu_{d,\beta}[\nabla g_\e]_\infty\|_{\Ld^2(\R^d)}^2\,\lesssim\,\e^2\mu_{d,\beta}(\tfrac1\e)^2\|\mu_{d,\beta}[\nabla g]_\infty\|_{\Ld^2(\R^d)}^2.\]
In the non-integrable case $\beta<d$, Lemma~\ref{lem:est-Hf} and the corrector estimates of~Lemma~\ref{lem:cor} rather lead to
\begin{eqnarray*}
T_\e&\lesssim&\expec{\big\|[\nabla z_{\e}+\phi^*\nabla g_\e]_2[\nabla\phi+\Id]_2\big\|_{\Ld^\frac{2d}{2d-\beta}(\R^d)}^2}\\
&\lesssim&\|[\nabla z_{\e}]_2\|_{\Ld^\frac{2d}{2d-\beta}(\R^d;\Ld^4(\Omega))}^2+\|\mu_{d,\beta}[\nabla g_\e]_\infty\|_{\Ld^\frac{2d}{2d-\beta}(\R^d)}^2,
\end{eqnarray*}
and we deduce as above
\[T_\e\,\lesssim\,\|\mu_{d,\beta}[\nabla g_\e]_\infty\|_{\Ld^\frac{2d}{2d-\beta}(\R^d)}^2\,\lesssim\,\e^2\mu_{d,\beta}(\tfrac1\e)^2\|\mu^*[\nabla g]_\infty\|_{\Ld^\frac{2d}{2d-\beta}(\R^d)}^2.\]
In the critical case $\beta=d$, the $\Ld^\frac{2d}{2d-\beta}$ norm is replaced by an $\Ld^2$ norm with logarithmic weight; the proof is then similar, appealing to the weighted version of the annealed Calder\'on-Zygmund estimate of~Proposition~\ref{prop:ann}.
\end{proof}

\section{(Non-)Degeneracy of the limiting covariance}\label{chap:deg}

In this section, we investigate the possible degeneracy of the limiting covariance structure. 
We only treat the symmetric setting, and we separately consider the integrable and non-integrable cases.
The non-symmetric setting is open.
We denote by $\Mcal$ the set of matrices $\bb\in\R^{d\times d}$ such that the boundedness and ellipticity properties~\eqref{f.56} are satisfied, that is, $|\bb\xi|\le|\xi|$ and $\xi\cdot\bb\xi\ge\lambda|\xi|^2$ for all $\xi\in\R^d$, and we denote by $\Mcal_\sym$ the subset of symmetric matrices in $\Mcal$.

\medskip
We start with the statements of the results: sufficient conditions for non-degeneracy and genericity of the non-degeneracy, both for
the case of integrable and non-integrable covariance. Proofs are postponed to the following subsections.

\medskip
In the integrable case $\beta>d$, recall that the effective fluctuation tensor $\calQ$ is defined in Proposition~\ref{prop:conv-cov}(i).
\begin{lem}\label{lem:non-deg}
Let $G$ be an $\R^\kappa$-valued Gaussian random field with an integrable covariance function $c$ that is of class $C^{2+\eta}$ in a neighborhood of the origin for some $\eta>0$,
and assume  
\begin{enumerate}
\renewcommand{\labelenumi}{\theenumi}
\renewcommand{\theenumi}{\textbf{(H\arabic{enumi})}}
\smallskip\item\label{pr-nondeg} \emph{Non-degeneracy of the covariance structure:}\\
If a stationary and centered random field $\psi\in\Ld^2(\Omega)^\kappa$ satisfies
\[\int_{\R^d}\expec{\psi_l(x)(1+\Lc)^{-1}\psi_m(0)}c_{lm}(x)\,dx\,=\,0,\]
then $\psi \equiv 0$.
\end{enumerate}
 Let $\Aa=a_0(G)$ with $a_0\in C^1_b(\R^\kappa;\Mcal_\sym)$.
If there exist $y,\alpha \in \R^\kappa$ such that the symmetric matrix $\alpha_l \partial_la_0(y)$ is definite, then $\calQ_{iiii}\ne 0$ for all~$1\le i\le d$.
\end{lem}

\noindent
Note that Property~\ref{pr-nondeg} trivially holds true if the Fourier transform $\hat c$ is pointwise positive, which is in particular compatible with the choice~\eqref{eq:cov-as}, and indeed provides many examples. 
Here comes the short argument. Setting $\Psi:=(1+\Lc)^{-1/2}\psi$ and $c_\Psi(x):=\expec{\Psi(x)\otimes\Psi(0)}$, the condition takes the form $\int_{\R^d}\hat c_\Psi:\hat c=0$ in Fourier space. Note that stationarity of $\Psi$ implies, for all $g\in C^\infty_c(\R^d)^\kappa$,
\[\int_{\R^d\times\R^d}g_l(x)g_m(y)\,(c_\Psi)_{lm}(x-y)\,dxdy=\expec{\Big|\int_{\R^d} g_l\Psi_l\Big|^2}\ge0,\]
hence Bochner's theorem ensures that the Fourier transform $\hat c_\Psi$ is a nonnegative measure. 
If $\hat c>0$ holds pointwise, the condition $\int_{\R^d}\hat c_\Psi:\hat c=0$ thus implies $\hat c_\Psi=0$, hence $\psi=0$, as claimed.
There is another trivial case when the property is satisfied.
As a consequence of an iterated use of the Helffer-Sjöstrand identity of Proposition~\ref{prop:Mall}(ii), it is also easily checked that Property~\ref{pr-nondeg} holds true when restricted to random fields of the form $\psi(x)=\psi_0(G(x))$ for a smooth function $\psi_0$; the corrector is of that special form in dimension $d=1$. 
We believe Property~\ref{pr-nondeg} might hold generically --- this constitutes an open question.

\medskip

The above condition for non-degeneracy is rather weak
and turns out to entail the generic non-degeneracy of the fluctuation tensor $\calQ$.
More precisely, given a Gaussian field $G$ with integrable covariance, there is a dense open set of transformations of the form $\Aa=a_0(G)$ that lead to a non-degenerate fluctuation tensor.
\begin{lem}\label{lem:non-deg-dense}
Let $G$ be an $\R^\kappa$-valued Gaussian random field with integrable covariance function, and let $s\ge0$.
For all $a_0\in C^\infty_b(\R^\kappa;\Mcal_\sym)$ there exists a sequence $(a_0^n)_n\subset C^\infty_b(\R^\kappa;\Mcal_\sym)$ such that $\Aa^n:=a_0^n(G)\to a_0(G)=:\Aa$ and $\partial^r a_0^n(G)\to \partial^r a_0(G)$
in $\Ld^p(\Omega)$ for all $r\in \N$ and $p<\infty$,
and such that for all $n$ the fluctuation tensor $\calQ^n$ associated with the coefficient field $\Aa^n$ is non-degenerate in the sense of $\calQ^n_{iiii}\ne0$ for all $1\le i\le d$.
For $s\ge1$, the convergence properties ensure $\bar\Aa^n\to\bar\Aa$ and $\calQ^n\to\calQ$.
\end{lem}

\medskip

In the non-integrable case $\beta<d$, by Proposition~\ref{prop:conv-cov}(iii), the fluctuation tensor field takes the form $\calQ_{ijij}(u):=\Kc_{ij}^l\Kc_{ij}^mC_{lm}(u)$. If for all $u$ the matrix $C(u)$ is positive definite (as would indeed follow from~\eqref{eq:cov-as}), the non-degeneracy of the fluctuation tensor field is equivalent to the non-vanishing of the tensor $\Kc$, for which the following trivial lemma establishes a sufficient condition.
\begin{lem}\label{lem:non-deg2}
Let $G$ be an $\R^\kappa$-valued Gaussian random field
and let $\Aa=a_0(G)$ with $a_0\in C^1_b(\R^\kappa;\Mcal_\sym)$.
Given $1\le l\le \kappa$, if the symmetric matrix $\partial_la_0(y)$ is definite for all $y\in\R^\kappa$, then $\Kc_{ii}^l\ne 0$ for all $1\le i\le d$.
\end{lem}
Although the above sufficient condition is much more stringent than in the integrable case, it still implies that non-degeneracy is a generic property.
\begin{lem}\label{lem:non-deg-dense2}
Let $G$ be an $\R^\kappa$-valued Gaussian random field, and let $s\ge1$.
For all $a_0\in C^s_b(\R^\kappa;\Mcal_\sym)$ there exists a sequence $(a_0^n)_n\subset C^s_b(\R^\kappa;\Mcal_\sym)$
such that $\Aa^n:=a_0^n(G)\to a_0(G)=:\Aa$ and $\partial^r a_0^n(G)\to \partial^ra_0(G)$ in $\Ld^\infty(\Omega)$ for all $0\le r\le s$, and such that the tensor $\Kc^n$ associated with $\Aa^n$ is non-degenerate in the sense of $(\Kc^n)_{ii}^l\ne0$ for all $1\le i\le d$ and $1\le l\le\kappa$. The convergence properties ensure in particular $\bar\Aa^n\to\bar\Aa$ and $\Kc^n\to\Kc$.
\end{lem}

\subsection{Integrable case}
We start with the proof of the sufficient condition for non-degeneracy given by Lemma~\ref{lem:non-deg}.
\begin{proof}[Proof of Lemma~\ref{lem:non-deg}]
In the integrable case with $\Aa$ symmetric, according to Proposition~\ref{prop:conv-cov}(i), the fluctuation tensor is defined by 
\begin{multline}\label{eq:formula-Q}
\calQ_{iji'j'}\,=\,\int_{\R^d}\E\Big[\big((\nabla\phi_j+\ee_j)\cdot \partial_la_0(G)(\nabla\phi_i+\ee_i)\big)(x)\\
(1+\Lc)^{-1}\big((\nabla\phi_{j'}+\ee_{j'})\cdot\partial_ma_0(G)(\nabla\phi_{i'}+\ee_{i'})\big)(0)\Big] c_{lm}(x)\,dx.
\end{multline}
By Property~\ref{pr-nondeg}, we see that the condition $\calQ_{iiii}=0$ holds for some~$i$ if and only if $(\nabla\phi_i+\ee_i)\cdot \partial_la_0(G)(\nabla\phi_i+\ee_i)\equiv0$ for all~$l$.
Since $a_0$ is of class $C^1$, there exists by assumption an open neighborhood $U\subset\R^\kappa$ of $y$ such that $\alpha_l \partial_l a_0$ is definite on~$U$.
In particular, the condition $(\nabla\phi_i+\ee_i)\cdot \alpha_l \partial_la_0(G)(\nabla\phi_i+\ee_i)\equiv0$ implies 
$\nabla\phi_i+\ee_i\equiv0$ conditioned on the event that $G\in U$.
Since the covariance function $c$ is continuous at the origin, we find $\pr{\forall x\in B:G(x)\in U}>0$, where $B$ denotes the unit ball of $\R^d$ at the origin.
Hence, if $\calQ_{iiii}=0$ holds for some $i$, we deduce $\pr{\forall x\in B:\nabla\phi_i(x)+\ee_i=0}>0$.
As the covariance function $c$ is assumed to be of class $C^{2+\eta}$ at the origin for some $\eta>0$, it follows e.g.\@ from Dudley's metric entropy bounds~\cite{Dudley-67} that $G$ (hence $\Aa$) is almost surely locally Lipschitz continuous.
We may then apply analytic continuation for $\Aa$-harmonic functions (cf.~\cite{Garofalo-Lin-86}), which upgrades the above into $\pr{\nabla\phi_i+\ee_i\equiv 0}>0$. By ergodicity, this implies $\nabla \phi_i+\ee_i\equiv 0$ almost surely,
which leads to $0=\expec{\nabla \phi_i+\ee_i}=e_i$, a contradiction.
\end{proof}
In particular, in the case when the coefficient field $\Aa$ is diagonal, we deduce the following simplified sufficient condition, which extends the non-degeneracy observation of~\cite{MO,GN}  to the continuum setting.
\begin{cor}\label{cor:diag}
Let $\Aa$ be a diagonal coefficient field of the form $\Aa_{ii}=a_{0,i}(G_i)$ for some $a_{0,i}\in C^\infty_b(\R;[\lambda,1])$ and some $\R$-valued Gaussian random fields $G_i$ with integrable covariance function. If the Gaussian field $G=(G_i)_{i=1}^d$ is non-degenerate and if for all~$i$ the function $a_{0,i}$ is not uniformly constant, then $\calQ_{iiii}\ne 0$ for all $1\le i\le d$.
\end{cor}

Next, we deduce that the non-degeneracy of the fluctuation tensor $\calQ$ is a generic property, as stated in Lemma~\ref{lem:non-deg-dense}.
\begin{proof}[Proof of Lemma~\ref{lem:non-deg-dense}]
Let $\chi\in C^\infty(\R)$ be nonnegative and compactly supported in $(-\frac12,\frac12)$ with $\chi'(0)=1$.
For all $n\ge1$, define $a_0^n(G):=a_0(G)+2\Id\chi(G_1-n)\sup|\partial a_0|$.
Since $G_1$ is Gaussian, we find $\partial^r a_0^n(G)\to \partial^r a_0(G)$ in $\Ld^p(\Omega)$ for all $r\in \N$ and $p<\infty$.
Denote by $\phi^n$ the corrector associated with $\Aa^n$.
Considering the corrector equation~\eqref{eq:corr} in the form 
\begin{equation}\label{e:diff-corr}
-\nabla \cdot \Aa \nabla (\phi^n-\phi)=\nabla \cdot (\Aa^n-\Aa) (\nabla \phi^n+ \Id),
\end{equation}
we deduce from the annealed Calder\'on-Zygmund estimate of Proposition~\ref{prop:ann} that $[\nabla\phi^n-\nabla\phi]_2\to0$ in $\Ld^p(\Omega)$ for all $p<\infty$,
which easily entails $\bar\Aa^n\to \bar\Aa$ and $\calQ^n\to \calQ$.
It remains to notice that $\partial_1 a_0^n(n\ee_1)$ is symmetric positive definite,
so that $\calQ^n$ is non-degenerate by Lemma~\ref{lem:non-deg}.
\end{proof}

\subsection{Non-integrable case}

We first check the sufficient condition for non-degeneracy given by Lemma~\ref{lem:non-deg2}.

\begin{proof}[Proof of Lemma~\ref{lem:non-deg2}]
By continuity of $\partial a_0$, the assumption ensures that $\partial_la_0(y)$ is either positive definite for all $y\in\R^\kappa$, or negative definite.
The conclusion then follows from the formula 
\[\Kc_{ii}^l=\expec{(\nabla\phi_{i}+\ee_i)\cdot (\partial_la_0)(G)(\nabla\phi_{i}+\ee_i)}.\qedhere\]
\end{proof}
This sufficient condition is particularly stringent compared to Lemma~\ref{lem:non-deg} since it requires definiteness at all points rather than at one single point.
This result is complemented with examples of non-degenerate and degenerate fluctuation tensors.
Note that the degenerate example~(ii) below is in sharp contrast with Corollary~\ref{cor:diag}, which indeed states that if $G$ had integrable covariance then even in the situation of~(ii) the corresponding fluctuation tensor would be non-degenerate for all~$z$.
\begin{lem}
Let $G$ be an $\R$-valued Gaussian random field ($\kappa=1$) and let $\Aa=a_0(G)\Id$ with $a_0\in C^1_b(\R;[\lambda,1])$.
\begin{enumerate}[(i)]
\item If $|a_0'|>0$ on $\R$, then $\Kc_{ii}^1\ne0$ for all $1\le i \le d$.
\smallskip\item If $\sup a_0=1$ and if $a_0(y)\to\lambda$ as $|y|\uparrow\infty$,
then there exists $z_0 \in \R$ such that the fluctuation tensor $\Kc^{z_0}$ of the shifted coefficient field $\Aa^{z_0}:= a_0(G+z_0)$ satisfies $(\Kc^{z_0})^1_{ii}=0$ for all $1\le i\le d$.
\qedhere
\end{enumerate}
\end{lem}

\begin{proof}
Item (i) is a direct consequence of Lemma~\ref{lem:non-deg2}.
We turn to~(ii), for which we start with a reformulation of~$\Kc_{ii}^1$.
For $z\in\R$, we consider the Gaussian field $G+z$, the corresponding coefficient field $\Aa^z:=a_0(G+z)\Id$,  we denote by $\phi^z$ the solution of the associated corrector equation (cf.~\eqref{eq:corr}),
\[-\nabla\cdot\Aa^z(\nabla\phi^z_i+\ee_i)=0,\]
and we denote by $\bar\Aa^z$ the homogenized coefficient associated with $\Aa^z$.
We may then compute
\begin{multline*}
\nabla_z (\bar\Aa^z)_{ii}|_{z=0}\,=\,\nabla_{z}\big(\expec{(\nabla\phi_i^z+\ee_i)\cdot\Aa^z(\nabla\phi_i^z+\ee_i)}\big)\big|_{z=0}\\
\,=\,\expec{(\nabla\phi_i+\ee_i)\cdot a_0'(G)(\nabla\phi_i+\ee_i)}+\expec{\nabla(\nabla_{z}\phi_{i}^z|_{z=0})\cdot\Aa(\nabla\phi_i+\ee_i)}\\
+\expec{(\nabla\phi_i+\ee_i)\cdot\Aa\nabla(\nabla_{z}\phi_j^z|_{z=0})}.
\end{multline*}
The first right-hand side term coincides with $\Kc_{ii}^1$ while the last two terms vanish due to the corrector equation~\eqref{eq:corr}, so that the above takes the form
\[\Kc_{ii}^1\,=\,\nabla_z (\bar\Aa^z)_{ii}|_{z=0}.\]
Note that these quantities do not depend on $i$ since $\Aa$ (hence $\bar\Aa^z$) is a multiple of the identity.
On the one hand, since by assumption $a_0(G+z)\to\lambda$ almost surely as $|z|\uparrow\infty$, we deduce $(\bar\Aa^z)_{ii}\to\lambda$ as $|z|\uparrow\infty$.
On the other hand, the standard harmonic lower bound for homogenized coefficients yields $(\bar\Aa^z)_{ii}>\lambda$ for all $z\in \R$.
By continuity in $z$, there exists $z_0 \in \R$ such that $(\bar\Aa^{z_0})_{ii}$ is maximal. 
Since the map $z\mapsto (\bar\Aa^z)_{ii}$ is obviously of class $C^1$, we deduce $\nabla_z (\bar\Aa^z)_{ii}|_{z=z_0}=0$, that is, $(\Kc^{z_0})_{ii}^1=0$.
\end{proof}

Next, we prove that the non-vanishing of the tensor $\Kc$ is a generic property, as stated in Lemma~\ref{lem:non-deg-dense2}.
\begin{proof}[Proof of Lemma~\ref{lem:non-deg-dense2}]
Using estimates on differences of correctors as in the proof of Lem\-ma~\ref{lem:non-deg-dense}, if $\Kc^l_{ii}\ne0$ and if approximations $\Aa^n:=a_0^n(G)$ satisfy $\partial^ra_0^n(G)\to\partial^ra_0(G)$ in $\Ld^p(\Omega)$ for all $0\le r\le 1$ and $p<\infty$, then the tensors $\Kc^n$ associated with $\Aa^n$ also satisfy $(\Kc^n)^l_{ii}\ne0$ for all $n$ large enough.
Therefore, it suffices to prove the result for $i=1$ and $l=1$, while the result for all $1\le i\le d$ and $1\le l\le\kappa$ follows by successive applications.
If $a_0$ is such that $\Kc_{11}^1\ne0$, there is nothing to prove.
Let $a_0\in C^s(\R^\kappa;\Mcal_\sym)$ be fixed with $\Kc_{11}^1=0$.
Let $b\in C^\infty_b(\R)$ be chosen with the following properties,
\begin{enumerate}[\qquad$\bullet$]
\item $b(y)=e^{-|y|}$ for $|y|>\frac12$;
\item $b$ is increasing on $(-\infty,0)$ and decreasing on $(0,\infty)$;
\item $b(y)\le e^{-|y|}$ and $|b'(y)|\le e^{-|y|}$ for all $y$;
\item $\big(\tfrac d{dy}\big)^sb|_{y=0}=0$ for all $s\ge1$.
\end{enumerate}
Next, for all $\eta>0$, we define the following asymmetric rescaling of $b$,
\[b^\eta(y)\,:=\,\left\{\begin{array}{lll}
\eta b(\tfrac1\eta y)&:&y<0,\\
\eta^2b(\tfrac1{\eta^2}y)&:&y\ge0,
\end{array}\right.\]
and we note that $b^\eta\in C^\infty_b(\R)$.
For $\eta>0$, $z\in\R$, and $n\ge1$, we then consider the following perturbations of $\Aa=a_0(G)$,
\[\Aa^{\eta,z,n}\,:=\,\Aa+\tfrac1n\, b^{\eta}(G_1-z)\Id,\]
as well as the associated correctors $\phi^{\eta,z,n}$ and tensors $\Kc^{\eta,z,n}$.
Expanding the perturbation and using energy estimates for differences of correctors~\eqref{e:diff-corr}, we find
\begin{multline*}
(\Kc^{\eta,z,n})_{11}^1\,=\,\Kc_{11}^1+2\, \expec{(\nabla \phi_1+\ee_1)\cdot \partial_la_0(G) \nabla(\phi^{\eta,z,n}_1-\phi_1)}
\\
+\tfrac1n\, \expec{|\nabla \phi_1+\ee_1|^2 (b^{\eta})'(G_1-z)}+O_{\alpha,\eta,z}(\tfrac1{n^2}).
\end{multline*}
Recalling the assumption that $\Kc_{11}^1=0$,
and using again energy estimates for differences of correctors~\eqref{e:diff-corr} in the form
\begin{multline*}
\quad\big|\expec{(\nabla \phi_1+\ee_1)\cdot \partial_la_0(G) \nabla (\phi^{\eta,z,n}_1-\phi_1)}\big|
\,\lesssim\,\expec{|\nabla (\phi^{\eta,z,n}_1-\phi_1)|^2}^\frac12
\\
\,\lesssim\,\tfrac1n\expec{|\nabla\phi_1+e_1|^2|b^{\eta}(G_1-z)|^2}^\frac12,
\end{multline*}
we deduce
\begin{multline}\label{eq:rewr-K111}
(\Kc^{\eta,z,n})_{11}^1\,\ge\,\tfrac1n\, \expec{|\nabla \phi_1+\ee_1|^2\,(b^{\eta})'(G_1-z)}\\
-\tfrac Cn\,\expec{|\nabla\phi_1+e_1|^2 |b^{\eta}(G_1-z)|^2}^\frac12-C_{\alpha,\eta,z}\tfrac1{n^2}.
\end{multline}
We now argue that we can choose $0<\eta\le 1$ and $z\in\R$ such that $(\Kc^{\eta,z,n})_{11}^1$ is nonzero for all $n$ large enough.
The construction of the suitable choice of $\eta,z$ is split into four steps:

\begin{enumerate}[$\bullet$]
\smallskip\item Since $\expec{|\nabla \phi_1+\ee_1|^2}\simeq1$, it is easily seen by conditioning and by continuity in $z$ that there exist $z_0\in \R$ and $0<\eta_0\ll 1$ such that 
\begin{eqnarray}\label{e.choice-z}
\quad\gamma&:=&\inf_{z:|z-z_0|<\eta_0}\expeC{|\nabla \phi_1(0)+\ee_1|^2}{G_1(0)=z}~>~0,\\
\quad\Gamma&:=&\sup_{z:|z-z_0|<\eta_0}\expeC{|\nabla \phi_1(0)+\ee_1|^2}{G_1(0)=z}~<~\infty.\nonumber
\end{eqnarray}

\smallskip\item
We show that there exists $\delta\simeq1$ such that
$$
\lim_{\eta \downarrow 0} \frac{\expec{|\nabla\phi_1+e_1|^2 |b^{\eta}(G_1-z_0)|^2}^\frac12}{\eta^\delta\,\expec{|\nabla\phi_1+e_1|^2 |(b^{\eta})'(G_1-z_0)\big|}}\,=\,0,
$$
hence, for $\eta>0$ small enough,
\begin{equation}\label{eq:plop2}
\quad\expec{|\nabla\phi_1+e_1|^2 |b^{\eta}(G_1-z_0)|^2}^\frac12\,\le\eta^\delta\,\expec{|\nabla\phi_1+e_1|^2 |(b^{\eta})'(G_1-z_0)|}.
\end{equation}
By definition of $b^\eta$, using the Meyers integrability of the correctors (cf. Lemma~\ref{si}), the numerator is estimated as follows: there exists $\delta\simeq1$ such that, for all $0<\eta\le1$,
\begin{multline*}
\quad\expec{|\nabla\phi_1+e_1|^2|b^{\eta}(G_1-z_0)|^2}^\frac12\,\le\,\eta\,\expec{|\nabla\phi_1+e_1|^2e^{-\frac2\eta|G_1-z_0|}}^\frac12\\
\,\le\,\eta\,\expec{|\nabla\phi_1+e_1|^\frac2{1-4\delta}}^{\frac{1-4\delta}2}\expec{e^{-\frac1{2\eta\delta}|G_1-z_0|}}^{2\delta}
\,\lesssim_{\alpha,z_0}\,\eta^{1+2\delta},
\end{multline*}
while for the denominator we deduce from~\eqref{e.choice-z}, for all $0<\eta <\eta_0$,
\begin{multline*}
\quad\expec{|\nabla \phi_1+\ee_1|^2 |(b^{\eta})'(G_1-z_0)|}
\\
\,\ge\, \expeC{|\nabla \phi_1+\ee_1|^2 |(b^{\eta})'(G_1-z_0)|}{\tfrac12\eta<z_0-G_1<\eta} \pr{\tfrac12\eta<z_0-G_1<\eta}
\\
\,\ge\, e^{-1}\expeC{|\nabla \phi_1+\ee_1|^2}{\tfrac12\eta<z_0-G_1<\eta} \pr{\tfrac12\eta<z_0-G_1<\eta}\,\gtrsim_{z_0}\, \gamma  \eta,
\end{multline*}
and the claim follows.

\smallskip\item
We show that
\[\limsup_{\eta\downarrow0}\frac{\expec{|\nabla\phi_1+\ee_1|^2|(b^{\eta})'(G_1-z_0)|}}{\expec{|\nabla\phi_1+\ee_1|^2\,(b^{\eta})'(G_1-z_0)}}\,\lesssim_{z_0,\gamma,\Gamma}\,1,\]
hence, for $\eta>0$ small enough,
\begin{equation}\label{eq:plop1}
\expec{|\nabla\phi_1+\ee_1|^2|(b^{\eta})'(G_1-z_0)|}\,\lesssim_{z_0,\gamma,\Gamma}\,\,\expec{|\nabla\phi_1+\ee_1|^2\,(b^{\eta})'(G_1-z_0)}.
\end{equation}
Denoting by $\rho$ the (Gaussian) law of $G_1(0)$, we compute
\begin{multline*}
\quad\expec{|\nabla\phi_1+\ee_1|^2|(b^{\eta})'(G_1-z_0)|\,\mathds1_{G_1\ge z_0}}\,\le\,\expec{|\nabla\phi_1+\ee_1|^2e^{-\frac1{\eta^2}|G_1-z_0|}}\\
\,\le\,\int_{|z-z_0|\le\eta_0}\expeC{|\nabla\phi_1+\ee_1|^2}{G_1=z}e^{-\frac1{\eta^2}|z-z_0|}d\rho(z)+C e^{-\eta_0\frac{1}{\eta^2}},
\end{multline*}
hence, appealing to~\eqref{e.choice-z},
\begin{equation*}
\quad\expec{|\nabla\phi_1+\ee_1|^2|(b^{\eta})'(G_1-z_0)|\mathds1_{G_1\ge z_0}}
\,\lesssim_{z_0}\,\eta^2\Gamma+e^{-\eta_0\frac{1}{\eta^2}},
\end{equation*}
and similarly,
\begin{equation*}
\quad\expec{|\nabla\phi_1+\ee_1|^2\,(b^{\eta})'(G_1-z_0)\mathds1_{G_1\le z_0}}\,\lesssim_{z_0}\,\eta\Gamma+e^{-\eta_0\frac1{\eta}}.
\end{equation*}
Conversely, for $0<\eta<\eta_0$, we deduce from~\eqref{e.choice-z},
\begin{multline*}
\quad\expec{|\nabla\phi_1+\ee_1|^2\,(b^{\eta})'(G_1-z_0)\,\mathds1_{G_1\le z_0}}\\
\,\ge\,\expeC{|\nabla\phi_1+\ee_1|^2\,(b^{\eta})'(G_1-z_0)}{\tfrac12\eta\le z_0-G_1\le\eta}\pr{\tfrac12\eta\le z_0-G_1\le\eta}\\
\,\ge\,e^{-1}\expeC{|\nabla\phi_1+\ee_1|^2}{\tfrac12\eta\le z_0-G_1\le\eta}\pr{\tfrac12\eta\le z_0-G_1\le\eta}\,\gtrsim_{z_0}\,\gamma\eta.
\end{multline*}
Combining these estimates in the form
\[\expec{|\nabla\phi_1+\ee_1|^2|(b^{\eta})'(G_1-z_0)|}\,\lesssim_{z_0}\,\eta\Gamma+e^{-\eta_0\frac1\eta},\]
and
\[\expec{|\nabla\phi_1+\ee_1|^2(b^{\eta})'(G_1-z_0)}\,\ge\,\eta(\tfrac\gamma{C_{z_0}}-C_{z_0}\eta\Gamma)-Ce^{-\eta_0\frac1{\eta^2}},\]
the claim follows.

\smallskip\item
The combination of the above observations shows that for $\eta>0$ small enough there holds for all $n\ge1$,
\begin{eqnarray*}
\quad(\Kc^{\eta,z_0,n})_{11}^1&\stackrel{\eqref{eq:rewr-K111}}{\ge}&\tfrac1n\, \expec{|\nabla \phi_1+\ee_1|^2\,(b^{\eta})'(G_1-z_0)}\\
&&\qquad-\tfrac Cn\,\expec{|\nabla\phi_1+e_1|^2 |b^{\eta}(G_1-z_0)|^2}^\frac12-C_{\alpha,\eta,z_0}\tfrac1{n^2}\\
&\stackrel{\eqref{eq:plop2}}{\ge}&\tfrac1n\, \expec{|\nabla \phi_1+\ee_1|^2\,(b^{\eta})'(G_1-z_0)}\\
&&\qquad-\eta^\delta\tfrac{C}n\expec{|\nabla\phi_1+e_1|^2|(b^{\eta})'(G_1-z_0)|}-C_{\alpha,\eta,z_0}\tfrac1{n^2}\\
&\stackrel{\eqref{eq:plop1}}{\ge}&\tfrac1{n}\big(\tfrac1{C_{z_0,\gamma,\Gamma}}-C\eta^\delta\big)\, \expec{|\nabla \phi_1+\ee_1|^2\,|(b^\eta)'(G_1-z_0)|}-C_{\alpha,\eta,z_0}\tfrac1{n^2}.
\end{eqnarray*}
Choosing $\eta>0$ small enough, the right-hand side is seen to be strictly positive for all $n$ large enough, and the conclusion follows.
\qedhere
\end{enumerate}
\end{proof}

\section{Asymptotic normality}\label{chap:normality}

In this section, we establish the asymptotic normality of the rescaled homogenization commutator, thus proving Theorem~\ref{th:main}(ii).
The proof is based on the second-order Poincaré inequality of Proposition~\ref{prop:Mall}(iii); in the integrable case $\beta>d$ we follow the argument of~\cite[Section~9]{DO1}.

\begin{proof}[Proof of Theorem~\ref{th:main}(ii)]
We focus on the case $\e=1$ and drop the subscript in the notation.
The final result will be obtained by rescaling in the last step of the proof.
Set $I(F):=\int_{\R^d}F:\Xi$.
We split the proof into six steps.

\medskip
\step1 Representation formula for Malliavin derivatives: We claim that
\begin{align}\label{eq:DJ0}
DI(F)\,=\,(F_{ij}\ee_j+\nabla S_i)\cdot\partial a_0(G)(\nabla\phi_i+\ee_i),
\end{align}
and
\begin{align}\label{eq:D2J0}
D^2I(F)\,=\,U_1+U_2+U_3,
\end{align}
in terms of
\begin{eqnarray*}
U_1(x,y)&:=&\delta(x-y)\big(F_{ij}(\nabla\phi_j^*+\ee_j)\cdot \partial^2a_0(G)(\nabla\phi_i+\ee_i)\big)(x),\\
U_2(x,y)&:=&\tilde U_2(x,y)+\tilde U_2(y,x),\\
\tilde U_2(x,y)&:=&\big(F_{ij}(\nabla\phi_j^*+\ee_j)\cdot\partial a_0(G)\nabla D_{y}\phi_i\big)(x),\\
U_3(x,y)&:=&\delta(x-y)\big(\phi_j^*\nabla F_{ij}\cdot \partial^2a_0(G)(\nabla\phi_i+\ee_i)\big)(x)\\
&&+\big(\phi_j^*\nabla F_{ij}\cdot\partial a_0(G)\nabla D_{y}\phi_i\big)(x)+\big(\phi_j^*\nabla F_{ij}\cdot\partial a_0(G)\nabla D_{x}\phi_i\big)(y)\\
&&\qquad+\int_{\R^d}\nabla F_{ij}\cdot(\Aa\phi_j^*+\sigma_j^*)\nabla D_{xy}^2\phi_i,
\end{eqnarray*}
where we identify the operators $U_i$ with their kernels and
where   the auxiliary field $S$ is the Lax-Milgram solution in $\R^d$ of
\begin{equation}\label{eq:aux-S}
-\nabla\cdot\Aa^*\nabla S_i\,=\,\nabla\cdot\big((\Aa-\bar\Aa)^*F_{ij}\ee_j\big).
\end{equation}
(Note that we use a very basic representation formula for the first Malliavin derivative, which is enough here as we only need to deduce the CLT scaling, whereas for the second Malliavin derivative a much finer decomposition is required.)

\noindent
We start with the proof of~\eqref{eq:DJ0}. We compute
\begin{eqnarray*}
DI(F)&=&D\int_{\R^d}F_{ij}\ee_j\cdot(\Aa-\bar\Aa)(\nabla\phi_i+\ee_i)\\
&=&\int_{\R^d}F_{ij}\ee_j\cdot D\Aa\,(\nabla\phi_i+\ee_i)+\int_{\R^d}F_{ij}\ee_j\cdot(\Aa-\bar\Aa)\,\nabla D\phi_i,
\end{eqnarray*}
hence, using the equation~\eqref{eq:aux-S} for $S$ and the corrector equation in the form~\eqref{eq:Dphii},
\begin{eqnarray*}
DI(F)&=&\int_{\R^d}F_{ij}\ee_j\cdot D\Aa\,(\nabla\phi_i+\ee_i)-\int_{\R^d}\nabla S_i\cdot\Aa\nabla D\phi_i\\
&=&\int_{\R^d}(F_{ij}\ee_j+\nabla S_i)\cdot D\Aa\,(\nabla\phi_i+\ee_i).
\end{eqnarray*}
Using~\eqref{eq:Da}, the conclusion~\eqref{eq:DJ0} follows.
We turn to~\eqref{eq:D2J0}. The second Malliavin derivative takes the form
\begin{multline}\label{eq:D2I-interm}
D^2_{xy}I(F)\,=\,D_y\int_{\R^d}F_{ij}\ee_j\cdot\big(D_x\Aa\,(\nabla\phi_i+\ee_i)+(\Aa-\bar\Aa)\nabla D_x\phi_i\big)\\
\,=\,\int_{\R^d}F_{ij}\ee_j\cdot\big(D^2_{xy}\Aa\,(\nabla\phi_i+\ee_i)+D_x\Aa\nabla D_y\phi_i+D_y\Aa\nabla D_x\phi_i+(\Aa-\bar\Aa)\nabla D^2_{xy}\phi_i\big),
\end{multline}
and it remains to reformulate the last RHS term.
Inserting the definition of the flux corrector $\sigma_j^*$ in the form $(\Aa^*-\bar\Aa^*)\ee_j=-\Aa^*\nabla\phi_j^*+\nabla\cdot\sigma_j^*$ and using the skew-symmetry of~$\sigma_j^*$,
we find
\begin{eqnarray*}
\int_{\R^d}F_{ij}\ee_j\cdot (\Aa-\bar\Aa)\nabla D^2_{xy}\phi_i&=&\int_{\R^d}F_{ij}(\Aa^*-\bar\Aa^*)\ee_j\cdot \nabla D^2_{xy}\phi_i\\
&=&\int_{\R^d}F_{ij}(\nabla\cdot\sigma_j^*)\cdot \nabla D^2_{xy}\phi_i-\int_{\R^d}F_{ij}\nabla\phi_j^*\cdot \Aa\nabla D^2_{xy}\phi_i\\
&=&\int_{\R^d}\nabla F_{ij}\cdot\sigma_j^* \nabla D^2_{xy}\phi_i-\int_{\R^d}F_{ij}\nabla\phi_j^*\cdot \Aa\nabla D^2_{xy}\phi_i.
\end{eqnarray*}
Taking the Malliavin derivative $D_y$ of~\eqref{eq:Dphii} yields
\begin{equation}\label{eq:D2phi}
-\nabla\cdot\Aa\nabla D^2_{xy}\phi_i=\nabla\cdot D_{xy}^2\Aa(\nabla\phi_i+\ee_i)+\nabla\cdot D_{x}\Aa \nabla D_y\phi_i+\nabla\cdot D_{y}\Aa \nabla D_x\phi_i,
\end{equation}
and we deduce
\begin{multline*}
\int_{\R^d}F_{ij}\ee_j\cdot (\Aa-\bar\Aa)\nabla D^2_{xy}\phi_i
\,=\,\int_{\R^d}\phi_j^*\nabla F_{ij}\cdot D_{xy}^2\Aa(\nabla\phi_i+\ee_i)
+\int_{\R^d}F_{ij}\nabla\phi_j^*\cdot D_{xy}^2\Aa(\nabla\phi_i+\ee_i)\\
+\int_{\R^d}\phi_j^*\nabla F_{ij}\cdot \big(D_{x}\Aa \nabla D_y\phi_i+D_y\Aa \nabla D_x\phi_i\big)
+\int_{\R^d}F_{ij}\nabla\phi_j^*\cdot \big(D_{x}\Aa \nabla D_y\phi_i+D_y\Aa \nabla D_x\phi_i\big)\\
+\int_{\R^d}\nabla F_{ij}\cdot(\Aa\phi_j^*+\sigma_j^*)\nabla D^2_{xy}\phi_i.
\end{multline*}
Inserting this into~\eqref{eq:D2I-interm}, and using~\eqref{eq:Da} and
\begin{equation}\label{eq:D2a}
D^2_{xy}\Aa=\partial^2a_0(G(x))\,\delta(\cdot-x)\delta(x-y),
\end{equation}
the conclusion~\eqref{eq:D2J0} follows.

\medskip
\step2 Proof of
\begin{align*}
\expec{\|DI(F)\|_{\Hf}^4}^\frac14\,\lesssim\,\left\{\begin{array}{lll}
\|F\|_{\Ld^2(\R^d)}^2&:&\beta>d,\\
\|\log(2+|\cdot|)^\frac12[F]_2\|_{\Ld^2(\R^d)}^2&:&\beta=d,\\
\|[F]_2\|_{\Ld^{\frac{2d}{2d-\beta}}(\R^d)}^2&:&\beta<d.
\end{array}\right.
\end{align*}
We only treat the non-integrable case $\beta<d$  (the other cases are treated similarly) and we appeal 
to the representation formula~\eqref{eq:DJ0} for the Malliavin derivative $DI(F)$.
Using Lemma~\ref{lem:est-Hf} and the corrector estimates of Lemma~\ref{lem:cor}, we find
\begin{eqnarray*}
\expec{\|D I(F)\|_{\Hf}^4}^\frac14&\lesssim&
\|[(F_{ij}\ee_j+\nabla S_i)\cdot\partial a_0(G)(\nabla\phi_i+\ee_i)]_1\|_{\Ld^\frac{2d}{2d-\beta}(\R^d;\Ld^4(\Omega))}\\
&\lesssim&
\|[F]_2\|_{\Ld^\frac{2d}{2d-\beta}(\R^d)}+\|[\nabla S]_2\|_{\Ld^\frac{2d}{2d-\beta}(\R^d;\Ld^8(\Omega))},
\end{eqnarray*}
and the conclusion follows from the annealed Calder\'on-Zygmund estimate of Proposition~\ref{prop:ann}.

\medskip
\step3 Proof that for all $p\ge4$,
\[\expec{\|U_1\|_\op^4}^\frac14\,\lesssim\,\left\{\begin{array}{lll}
p\,w_c(p)^2\,\|w_c^2\,[F]_\infty\|_{\Ld^p(\R^d)}&:&\beta \ge d,\\
\|[F]_\infty\|_{\Ld^{\frac{d}{d-\beta}}(\R^d)}&:&\beta<d,
\end{array}\right.\]
where henceforth we set $w_c(x):=\log(2+|x|)^\frac12$ in the critical case $\beta=d$ and $w_c\equiv1$ otherwise.

\smallskip\noindent
Decomposing the covariance function as $c=c_0\ast c_0$ and noting that the norm of $\zeta$ in $\Hf$ coincides with the norm of $c_0\ast\zeta$ in $\Ld^2(\R^d)$,
the definition~\eqref{eq:def-op} of the operator norm $\|\cdot\|_\op$ can be rewritten as follows,
\begin{eqnarray*}
\|U_1\|_\op&=&\sup_{\|c_0\ast\zeta\|_{\Ld^2(\R^d)}=\|c_0\ast\zeta'\|_{\Ld^2(\R^d)}=1}\Big|\int_{\R^d}\int_{\R^d}(c\ast\zeta)(x)\,(c\ast\zeta')(y)\,U_1(x,y)\,dxdy\Big|\\
&\lesssim&\sup_{\|c_0\ast\zeta\|_{\Ld^2(\R^d)}=1}\Big|\int_{\R^d}\int_{\R^d}(c\ast\zeta)(x)\,(c\ast\zeta)(y)\,U_1(x,y)\,dxdy\Big|.
\end{eqnarray*}
Further noting that by the Hardy-Littlewood-Sobolev inequality similarly as in Lemma~\ref{lem:est-Hf} the decay assumption~\eqref{eq:cov-as-bis} for $c_0$ implies
\[\|c_0\ast\zeta\|_{\Ld^2(\R^d)}\,\gtrsim\,
\left\{\begin{array}{lll}
\|w_c^{-1}[c\ast\zeta]_\infty\|_{\Ld^2(\R^d)}&:&\beta \ge d,\\
\|[c\ast\zeta]_\infty\|_{\Ld^{{2d}/\beta}(\R^d)}&:&\beta<d,
\end{array}\right.\]
we find
\begin{equation}
\|U_1\|_\op
\,\lesssim\,\left\{\begin{array}{lll}
\displaystyle\sup_{\|[\zeta]_\infty\|_{\Ld^2(\R^d)}=1}\Big|\int_{\R^d}\int_{\R^d}(w_c\zeta)(x)\,(w_c\zeta)(y)\,U_1(x,y)\,dxdy\Big|&:&\beta \ge d,\\
\displaystyle\sup_{\|[\zeta]_\infty\|_{\Ld^{{2d}/\beta}(\R^d)}=1}\Big|\int_{\R^d}\int_{\R^d}\zeta(x)\,\zeta(y)\,U_1(x,y)\,dxdy\Big|&:&\beta<d.
\end{array}\right.\label{eq:reform-op}
\end{equation}
In the integrable case $\beta\ge d$, for $p\ge4$, inserting the definition of $U_1$, using Hölder's inequality, and applying the discrete $\ell^{\frac{2p}{p-1}}$--$\ell^2$ inequality in the form
\[\|[\zeta]_\infty\|_{\Ld^{\frac{2p}{p-1}}(\R^d)}\lesssim\|[\zeta]_\infty\|_{\Ld^2(\R^d)},\]
we find
\begin{eqnarray*}
\|U_1\|_\op&\lesssim&\sup_{\|[\zeta]_\infty\|_{\Ld^2(\R^d)}=1}\|[\zeta]_\infty\|_{\Ld^{\frac{2p}{p-1}}(\R^d)}^2\big\|w_c^2\big[ F_{ij}(\nabla\phi^*_j+\ee_j)\cdot\partial^2a_0(G)(\nabla\phi_i+\ee_i)\big]_1\big\|_{\Ld^p(\R^d)}\\
&\lesssim&\|w_c^2\,[ F]_\infty[\nabla\phi^*+\Id]_2[\nabla\phi+\Id]_2\|_{\Ld^p(\R^d)},
\end{eqnarray*}
hence, by stationarity and by the corrector estimates of Lemma~\ref{lem:cor},
\begin{eqnarray}
\expec{\|U_1\|_\op^4}^\frac14&\lesssim&\|w_c^2\,[F]_\infty[\nabla\phi^*+\Id]_2[\nabla\phi+\Id]_2\|_{\Ld^p(\R^d;\Ld^p(\Omega))}\nonumber \\
&\lesssim&\|w_c^2\,[F]_\infty\|_{\Ld^p(\R^d)}\|[\nabla\phi^*+\Id]_2\|_{\Ld^{2p}(\Omega)}\|[\nabla\phi+\Id]_2\|_{\Ld^{2p}(\Omega)}\nonumber \\
&\lesssim& p\,w_c(p)^2\,\|w_c^2\,[F]_\infty\|_{\Ld^p(\R^d)}.\label{eq:bound-U1-int}
\end{eqnarray}
In the non-integrable case $\beta<d$, the corresponding estimates take the form
\begin{eqnarray*}
\expec{\|U_1\|_\op^4}^\frac14&\lesssim&\expec{\big\|\big[F_{ij}(\nabla\phi^*_j+\ee_j)\cdot\partial^2a_0(G)(\nabla\phi_i+\ee_i)\big]_1\big\|_{\Ld^{\frac d{d-\beta}}(\R^d)}^4}^\frac14\\
&\lesssim&\big\|[F]_\infty[\nabla\phi^*+\Id]_2[\nabla\phi+\Id]_2\big\|_{\Ld^{\frac d{d-\beta}}\big(\R^d;\Ld^{4\vee\frac d{d-\beta}}(\Omega)\big)}\\
&\lesssim&\|[F]_\infty\|_{\Ld^{\frac d{d-\beta}}(\R^d)},
\end{eqnarray*}
as claimed.

\medskip
\step4 Proof that for all $p\ge4$,
\[\expec{\|U_2\|_\op^4}^\frac14\,\lesssim\,
\left\{\begin{array}{lll}
p\,w_c(p)^2\,\|w_c^2\,[F]_\infty\|_{\Ld^p(\R^d)}&:&\beta\ge d,\\
\|[F]_\infty\|_{\Ld^\frac{d}{d-\beta}(\R^d)}&:&\beta<d.
\end{array}\right.\]
By symmetry, it suffices to estimate the norm of $\tilde U_2$.
We start with the integrable case $\beta\ge d$.  It follows from~\eqref{eq:reform-op} and Hölder's inequality that
\begin{equation*}
\|\tilde U_2\|_\op^2\,\lesssim\,\sup_{\|[\zeta]_\infty\|_{\Ld^2(\R^d)}=1}\int_{\R^d}w_c^2\Big[\int_{\R^d}(w_c\zeta)(x)\,\tilde U_2(x,\cdot)\,dx\Big]_1^2,
\end{equation*}
and hence, by duality in form of
\[\expec{\|\tilde U_2\|_\op^4}^\frac14\,=\,\sup_{\|\xi\|_{\Ld^4(\Omega)}=1}\expec{\xi^2\|\tilde U_2\|_\op^2}^\frac12,\]
we deduce
\begin{gather}\label{eq:rewr-U2-int}
\expec{\|\tilde U_2\|_\op^4}^\frac14
\,\lesssim\,\sup_{\|[\zeta]_\infty\|_{\Ld^4(\Omega;\Ld^2(\R^d))}=1}M(\zeta),\\
M(\zeta)\,:=\,\expec{\int_{\R^d}w_c^2\Big[\int_{\R^d}(w_c\zeta)(x)\,\tilde U_2(x,\cdot)\,dx\Big]_1^2}^\frac12.\nonumber
\end{gather}
Let $\zeta$ be fixed with $\|[\zeta]_\infty\|_{\Ld^4(\Omega;\Ld^2(\R^d))}=1$. Note that the discrete $\ell^r$--$\ell^2$  inequality and Jensen's inequality entail for all $2\le r\le4$,
\begin{equation}\label{eq:LrLr-1}
\|[\zeta]_\infty\|_{\Ld^r(\R^d;\Ld^r(\Omega))}\lesssim1.
\end{equation}
Inserting the definition of $\tilde U_2$, defining the auxiliary field $T_i$ as the unique Lax-Milgram solution of
\begin{equation}\label{eq:def-aux-Ti}
-\nabla\cdot\Aa^*\nabla T_i=\nabla\cdot\big(w_c\,\zeta\,F_{ij}\,\partial a_0^*(G)(\nabla\phi_j^*+\ee_j)\big),
\end{equation}
using the corrector equation for $\phi_i$ in the form~\eqref{eq:Dphii}, and using~\eqref{eq:Da},
we may write
\begin{eqnarray}
M(\zeta)&=&\expec{\int_{\R^d}w_c^2\Big[\int_{\R^d}w_c\,\zeta\,F_{ij}\,(\nabla\phi_j^*+\ee_j)\cdot\partial a_0(G)\nabla D\phi_i\Big]_1^2}^\frac12 \nonumber \\
&=&\expec{\int_{\R^d}w_c^2\Big[\int_{\R^d}\nabla T_i\cdot\Aa\nabla D\phi_i\Big]_1^2}^\frac12 \nonumber\\
&=&\expec{\int_{\R^d}w_c^2\Big[\int_{\R^d}\nabla T_i\cdot D\Aa(\nabla\phi_i+\ee_i)\Big]_1^2}^\frac12 \nonumber\\
&\lesssim&\expec{\int_{\R^d}w_c^2\,[\nabla T]_2^2\,[\nabla\phi+\Id]_2^2}^\frac12.\label{eq:boundM-1step}
\end{eqnarray}
By Hölder's inequality and the corrector estimates of Lemma~\ref{lem:cor}, for $p\ge4$, this entails
\begin{eqnarray*}
M(\zeta)
\,\lesssim\,p^\frac12w_c(p)\|w_c[\nabla T]_2\|_{\Ld^2\big(\R^d;\Ld^\frac{4p}{2p-1}(\Omega)\big)}.
\end{eqnarray*}
Applying the annealed Calder\'on-Zygmund estimate of Proposition~\ref{prop:ann} with logarithmic weight and integrability loss $\delta=\frac{2p}{(p-1)(2p-1)}\sim\frac1p$, using Hölder's inequality, the corrector estimates of Lemma~\ref{lem:cor}, and~\eqref{eq:LrLr-1}, we are led to
\begin{eqnarray*}
M(\zeta)
&\lesssim&p^\frac12w_c(p)\|w_c^2\,[\zeta]_\infty[F]_\infty[\nabla\phi^*+\Id]_2\|_{\Ld^2\big(\R^d;\Ld^{\frac{2p}{p-1}}(\Omega)\big)}\\
&\lesssim&p^\frac12w_c(p)\|[\zeta]_\infty\|_{\Ld^{\frac{2p}{p-2}}\big(\R^d;\Ld^{\frac{2p}{p-2}}(\Omega)\big)}\|[\nabla\phi^*+\Id]_2\|_{\Ld^{2p}(\Omega)}\|w_c^2\,[F]_\infty\|_{\Ld^p(\R^d)}\\
&\lesssim&p\,w_c(p)^2\,\|w_c^2\,[F]_\infty\|_{\Ld^p},
\end{eqnarray*}
and the conclusion follows.

\smallskip\noindent
We turn to the non-integrable case $\beta<d$.
It follows from~\eqref{eq:reform-op} and Hölder's inequality that
\begin{equation*}
\|\tilde U_2\|_\op\,\lesssim\,\sup_{\|[\zeta]_\infty\|_{\Ld^{2d/\beta}(\R^d)}=1}\bigg(\int_{\R^d}\Big[\int_{\R^d}\zeta(x)\,\tilde U_2(x,\cdot)\,dx\Big]_1^\frac{2d}{2d-\beta}\bigg)^{\frac{2d-\beta}{2d}},
\end{equation*}
and hence, by duality in form of
\[\expec{\|\tilde U_2\|_\op^4}^\frac14=\sup_{\|\xi\|_{\Ld^{4d/\beta}(\Omega)}=1}\expec{\xi^\frac{4d}{d+\beta}\|\tilde U_2\|_\op^\frac{4d}{d+\beta}}^\frac{d+\beta}{4d},\]
we deduce the following version of~\eqref{eq:rewr-U2-int},
\begin{gather}\label{eq:rewr-U2-noint}
\expec{\|\tilde U_2\|_\op^4}^\frac14
\,\lesssim\,\sup_{\|[\zeta]_\infty\|_{\Ld^{4d/\beta}(\Omega;\Ld^{2d/\beta}(\R^d))}=1}M(\zeta),\\
M(\zeta)\,:=\,\expec{\bigg(\int_{\R^d}\Big[\int_{\R^d}\zeta(x)\,\tilde U_2(x,y)\,dx\Big]_1^{\frac{2d}{2d-\beta}}dy\bigg)^{2\frac{2d-\beta}{d+\beta}}}^\frac{d+\beta}{4d}.\nonumber
\end{gather}
Let $\zeta$ be fixed with $\|[\zeta]_\infty\|_{\Ld^{4d/\beta}(\Omega;\Ld^{{2d}/\beta}(\R^d))}=1$.
Note that the discrete $\ell^r$--$\ell^\frac{2d}{\beta}$ inequality and Jensen's inequality entail for all $\frac{2d}\beta\le r\le \frac{4d}{\beta}$,
\begin{equation}\label{eq:LrLr-2}
\|[\zeta]_\infty\|_{\Ld^r(\R^d;\Ld^r(\Omega))}\lesssim1.
\end{equation}
Arguing as in~\eqref{eq:boundM-1step}, with the auxiliary field $T_i$ defined in~\eqref{eq:def-aux-Ti}, and using the the triangle inequality with $2\frac{2d-\beta}{d+\beta}>1$, we obtain
\begin{eqnarray*}
M(\zeta)&\lesssim&\expec{\bigg(\int_{\R^d}[\nabla T]_2^\frac{2d}{2d-\beta}[\nabla\phi+\Id]_2^\frac{2d}{2d-\beta}\bigg)^{2\frac{2d-\beta}{d+\beta}}}^\frac{d+\beta}{4d}\\
&\lesssim&\big\|[\nabla T]_2[\nabla\phi+\Id]_2\big\|_{\Ld^\frac{2d}{2d-\beta}\big(\R^d;\Ld^\frac{4d}{d+\beta}(\Omega)\big)}.
\end{eqnarray*}
For $s_0:=\frac{2d(d+5\beta)}{3\beta(d+\beta)}>\frac{4d}{d+\beta}$ the corrector estimates of Lemma~\ref{lem:cor} then yield
\begin{equation*}
M(\zeta)\,\lesssim\,\|[\nabla T]_2\|_{\Ld^\frac{2d}{2d-\beta}(\R^d;\Ld^{s_0}(\Omega))}.
\end{equation*}
For $s_0<s_1:=\frac{2d(2d+4\beta)}{3\beta(d+\beta)}<\frac{2d}{\beta}$ and for $\frac1{s_2}=\frac1{s_1}-\frac\beta{2d}$, applying the annealed Calder\'on-Zygmund estimates of Proposition~\ref{prop:ann}, the corrector estimates of Lemma~\ref{lem:cor}, and~\eqref{eq:LrLr-2}, we deduce
\begin{eqnarray*}
M(\zeta)&\lesssim&\|[\zeta]_\infty[F]_\infty[\nabla\phi^*+\Id]_2\|_{\Ld^\frac{2d}{2d-\beta}(\R^d;\Ld^{s_1}(\Omega))}\\
&\le&\|[\zeta]_\infty\|_{\Ld^\frac{2d}{\beta}\big(\R^d;\Ld^{\frac{2d}\beta}(\Omega)\big)}\|[\nabla\phi^*+\Id]_2\|_{\Ld^{s_2}(\Omega)}\|[F]_\infty\|_{\Ld^\frac{d}{d-\beta}(\R^d)}\\
&\lesssim&\|[F]_\infty\|_{\Ld^\frac{d}{d-\beta}(\R^d)},
\end{eqnarray*}
and the conclusion follows.

\medskip
\step5 Proof that for all $q\ge4$,
\[\expec{\|U_3\|_\op^4}^\frac14\,\lesssim\,\left\{\begin{array}{lll}
C_q\|\mu_{d,\beta}w_c^2\,[\nabla F]_\infty\|_{\Ld^q(\R^d)}&:&\beta\ge d,\\
\|\mu_{d,\beta}[\nabla F]_\infty\|_{\Ld^\frac d{d-\beta}(\R^d)}&:&\beta<d.
\end{array}\right.\]
Note that the dependence on $q$ does not need to be made specific here since this contribution is of higher order, cf.~Step~6.

\smallskip\noindent
We start with a suitable reformulation of $U_3$.
Defining the auxiliary field $V_i$ as the unique Lax-Milgram solution of
\[-\nabla\cdot\Aa^*\nabla V_i\,=\,\nabla\cdot\big((\Aa^*\phi_j^*-\sigma_j^*)\nabla F_{ij}\big),\]
using the corrector equation for $\phi_i$ in the form~\eqref{eq:D2phi}, and using~\eqref{eq:Da} and~\eqref{eq:D2a},
we may write
\begin{eqnarray*}
\lefteqn{\int_{\R^d}\nabla F_{ij}\cdot(\Aa\phi_j^*+\sigma_j^*)\nabla D_{xy}^2\phi_i\,=\,-\int_{\R^d}\nabla V_i\cdot\Aa\nabla D_{xy}^2\phi_i}\\
&\qquad=&\int_{\R^d}\nabla V_i\cdot D_{xy}^2\Aa\,(\nabla\phi_i+\ee_i)+\int_{\R^d}\nabla V_i\cdot\big(D_{x}\Aa\,\nabla D_y\phi_i+ D_{y}\Aa\,\nabla D_x\phi_i\big)\\
&\qquad=&\delta(x-y)\big(\nabla V_i\cdot \partial^2a_0(G)(\nabla\phi_i+\ee_i)\big)(x)\\
&\qquad&\qquad+\big(\nabla V_i\cdot\partial a_0(G)\nabla D_y\phi_i\big)(x)+\big(\nabla V_i\cdot\partial a_0(G)\nabla D_x\phi_i\big)(y).
\end{eqnarray*}
This allows to decompose $U_3=U_3^1+U_3^2$ with
\begin{eqnarray*}
U_3^1(x,y)&:=&\delta(x-y)\big((\phi_j^*\nabla F_{ij}+\nabla V_i)\cdot \partial^2a_0(G)(\nabla\phi_i+\ee_i)\big)(x),\\
U_3^2(x,y)&:=&\tilde U_3^2(x,y)+\tilde U_3^2(y,x),\\
\tilde U_3^2(x,y)&:=&\big((\phi_j^*\nabla F_{ij}+\nabla V_i)\cdot\partial a_0(G)\nabla D_{y}\phi_i\big)(x).
\end{eqnarray*}
As $U_3^1$ and $U_3^2$ have a similar structure as $U_1$ and $U_2$, their norms are estimated by a simple modification of the argument of Steps~3 and~4.
As an illustration, we treat $U_3^1$ in the integrable case $\beta\ge d$ --- the other estimates are analogous and details are omitted.
Arguing as in~\eqref{eq:bound-U1-int}, we find for $q\ge1$,
\begin{eqnarray*}
\expec{\|U_3^1\|_\op^4}^\frac14&\lesssim&\|w_c^2\,[\phi^*\nabla F+\nabla V]_2\|_{\Ld^q(\R^d;\Ld^{2q}(\Omega))}\|[\nabla\phi+\Id]_2\|_{\Ld^{2q}(\Omega)}\\
&\lesssim_q&\|w_c^2\,[\phi^*\nabla F+\nabla V]_2\|_{\Ld^q(\R^d;\Ld^{2q}(\Omega))},
\end{eqnarray*}
and the (weighted) annealed Calder\'on-Zygmund estimate of Proposition~\ref{prop:ann} (with $\delta=1$) together with the corrector estimates of Lemma~\ref{lem:cor} lead to
\begin{equation*}
\expec{\|U_3^1\|_\op^4}^\frac14\,\lesssim_q\,\|w_c^2\,[(\phi^*,\sigma^*)\nabla F]_2\|_{\Ld^q(\R^d;\Ld^{2q+1}(\Omega))}\,\lesssim_p\,\|\mu_{d,\beta}w_c^2\,[\nabla F]_\infty\|_{\Ld^q(\R^d)}.
\end{equation*}

\medskip
\step6 Conclusion.\\
In the integrable case $\beta > d$, for $I_\e(F):=\e^{-\frac d2}\int_{\R^d}F(x):\Xi(\frac x\e)\,dx$, the conclusions of Steps~3--5 yield by scaling, for all $4\le p,q<\infty$,
\begin{equation*}
\expec{\|D^2I_\e(F)\|_\op^4}^\frac14\,\lesssim\,\e^{\frac d2}\Big(
p\,\e^{-\frac dp}\|[F]_\infty\|_{\Ld^p(\R^d)}+C_q\e^{1-\frac dq}\mu_{d,\beta}(\tfrac1\e)\|\mu_{d,\beta}[\nabla F]_\infty\|_{\Ld^q(\R^d)}\Big).
\end{equation*}
Hence,
choosing $p=\Log$ and $q=2d\vee p$, we deduce
\begin{equation*}
\expec{\|D^2I_\e(F)\|_\op^4}^\frac14\,\lesssim\,\e^{\frac d2}\Log\Big(
\|[F]_\infty\|_{\Ld^2\cap \Ld^\infty(\R^d)}+\|\mu_{d,\beta}[\nabla F]_\infty\|_{\Ld^2\cap\Ld^\infty(\R^d)}\Big).
\end{equation*}
In the critical case $\beta=d$, for $I_\e(F):=\e^{-\frac d2}\Log^{-\frac12} \int_{\R^d}F(x):\Xi(\frac x\e)\,dx$, the same argument yields
\begin{multline*}
\expec{\|D^2I_\e(F)\|_\op^4}^\frac14\,\lesssim\,\e^{\frac d2}\Log^\frac32\log\Log\\
\times\Big(
\|w_c^2\,[F]_\infty\|_{\Ld^2\cap \Ld^\infty(\R^d)}+\|\mu_{d,\beta}w_c^2\,[\nabla F]_\infty\|_{\Ld^2\cap\Ld^\infty(\R^d)}\Big).
\end{multline*}
In the non-integrable case $\beta<d$, for $I_\e(F):=\e^{-\frac\beta2}\int_{\R^d}F(x):\Xi(\frac x\e)\,dx$, the conclusions of Steps~3--5 yield by scaling,
\begin{eqnarray*}
\expec{\|D^2I_\e(F)\|_\op^4}^\frac14&\lesssim&\e^{\frac\beta2}\Big(\|[F]_\infty\|_{\Ld^\frac{d}{d-\beta}(\R^d)}+\e\mu_{d,\beta}(\tfrac1\e)\|\mu_{d,\beta}[\nabla F]_\infty\|_{\Ld^\frac{d}{d-\beta}(\R^d)}\Big)\\
&\lesssim&\e^{\frac\beta2}\Big(\|[F]_\infty\|_{\Ld^\frac{d}{d-\beta}(\R^d)}+\|\mu_{d,\beta}[\nabla F]_\infty\|_{\Ld^\frac{d}{d-\beta}(\R^d)}\Big).
\end{eqnarray*}
Likewise, the result of Step~2 yields
\[\expec{\|DI_\e(F)\|_{\Hf}^4}^\frac14\,\lesssim\,\left\{\begin{array}{lll}
\|w_c F\|_{\Ld^2(\R^d)}&:&\beta\ge d,\\
\|F\|_{\Ld^{\frac{2d}{2d-\beta}}(\R^d)}&:&\beta<d.
\end{array}\right.\]
Now applying Proposition~\ref{prop:Mall}(iii) in the form
\begin{multline*}
\dWW{\frac{I_\e(F)}{\var{I_\e(F)}^\frac12}}\Nc+\dTV{\frac{I_\e(F)}{\var{I_\e(F)}^\frac12}}\Nc\\
\,\lesssim\,\frac{\expec{\|D^2I_\e(F)\|_\op^4}^\frac14\expec{\|DI_\e(F)\|_\Hf^4}^\frac14}{\var{I_\e(F)}},
\end{multline*}
and inserting the above estimates for $\|D^2I_\e(F)\|_\op$ and $\|DI_\e(F)\|_\Hf$, the conclusion follows.\\
\end{proof}

\section*{Acknowledgments}
The authors thank  Ivan Nourdin  and Felix Otto for inspiring discussions.
The work of MD is financially supported by the CNRS-Momentum program.
Financial support of AG is acknowledged from the European Research Council under the European Community's Seventh Framework Programme (FP7/2014-2019 Grant Agreement QUANTHOM 335410).

\bibliographystyle{plain}

\begin{thebibliography}{10}

\bibitem{Armstrong-Daniel-16}
S.~N. Armstrong and J.-P. Daniel.
\newblock Calder\'on-{Z}ygmund estimates for stochastic homogenization.
\newblock {\em J. Funct. Anal.}, 270(1):312--329, 2016.

\bibitem{AKM2}
S.~N. {Armstrong}, T.~{Kuusi}, and J.-C. {Mourrat}.
\newblock The additive structure of elliptic homogenization.
\newblock {\em Invent. Math.}, 208:999--1154, 2017.

\bibitem{AKM-book}
S.~N. Armstrong, T.~Kuusi, and J.-C. Mourrat.
\newblock Quantitative stochastic homogenization and large-scale regularity.
\newblock Springer, Cham, 2019.

\bibitem{AS}
S.~N. Armstrong and C.~K. Smart.
\newblock Quantitative stochastic homogenization of convex integral
  functionals.
\newblock {\em Ann. Sci. \'Ec. Norm. Sup\'er. (4)}, 49(2):423--481, 2016.

\bibitem{BGMP-08}
G.~Bal, J.~Garnier, S.~Motsch, and V.~Perrier.
\newblock Random integrals and correctors in homogenization.
\newblock {\em Asymptot. Anal.}, 59(1-2):1--26, 2008.

\bibitem{BG-18}
P.~Bella and A.~Giunti.
\newblock Green's function for elliptic systems: moment bounds.
\newblock {\em Netw. Heterog. Media}, 13(1):155--176, 2018.

\bibitem{C2}
S.~Chatterjee.
\newblock Fluctuations of eigenvalues and second order {P}oincar\'e
  inequalities.
\newblock {\em Probab. Theory Related Fields}, 143(1-2):1--40, 2009.

\bibitem{Doukhan-94}
P.~Doukhan.
\newblock {\em Mixing}, volume~85 of {\em Lecture Notes in Statistics}.
\newblock Springer-Verlag, New York, 1994.

\bibitem{Dudley-67}
R.~M. Dudley.
\newblock The sizes of compact subsets of {H}ilbert space and continuity of
  {G}aussian processes.
\newblock {\em J. Functional Analysis}, 1:290--330, 1967.

\bibitem{DGO2}
M.~Duerinckx, A.~Gloria, and F.~Otto.
\newblock Robustness of the pathwise structure of fluctuations in stochastic
  homogenization.
\newblock   {\em  Probab. Theory Related Fields}, 178(1-2):531--566, 2020.

\bibitem{DGO1}
M.~Duerinckx, A.~Gloria, and F.~Otto.
\newblock The structure of fluctuations in stochastic homogenization.
\newblock Comm. Math. Phys., 377(1):259--306, 2020.

\bibitem{DO1}
M.~Duerinckx and F.~Otto.
\newblock Higher-order pathwise theory of fluctuations in stochastic
  homogenization.
\newblock Stoch. Partial Differ. Equ. Anal. Comput., 8(3):625--692, 2020.

\bibitem{Garofalo-Lin-86}
N.~Garofalo and F.-H. Lin.
\newblock Monotonicity properties of variational integrals, {$A_p$} weights and
  unique continuation.
\newblock {\em Indiana Univ. Math. J.}, 35(2):245--268, 1986.

\bibitem{GMa}
A.~Gloria and D.~Marahrens.
\newblock Annealed estimates on the {G}reen functions and uncertainty
  quantification.
\newblock {\em Ann. Inst. H. Poincar\'e Anal. Non Lin\'eaire},
  33(5):1153--1197, 2016.

\bibitem{GNO-quant}
A.~Gloria, S.~Neukamm, and F.~Otto.
\newblock Quantitative estimates in stochastic homogenization for correlated coefficient fields.
\newblock {\em Analysis \& PDE}, to appear.

\bibitem{GNO-reg}
A.~Gloria, S.~Neukamm, and F.~Otto.
\newblock A regularity theory for random elliptic operators.
\newblock {\em Milan J. Math.} 88(1):99--170,  2020.

\bibitem{GNO1}
A.~Gloria, S.~Neukamm, and F.~Otto.
\newblock Quantification of ergodicity in stochastic homogenization: optimal
  bounds via spectral gap on {G}lauber dynamics.
\newblock {\em Invent. Math.}, 199(2):455--515, 2015.

\bibitem{GN}
A.~Gloria and J.~Nolen.
\newblock A quantitative central limit theorem for the effective conductance on
  the discrete torus.
\newblock {\em Comm. Pure Appl. Math.}, 69(12):2304--2348, 2016.

\bibitem{GO4}
A.~Gloria and F.~Otto.
\newblock The corrector in stochastic homogenization: optimal rates, stochastic
  integrability, and fluctuations.
\newblock Preprint, arXiv:1510.08290.

\bibitem{GO1}
A.~Gloria and F.~Otto.
\newblock An optimal variance estimate in stochastic homogenization of discrete
  elliptic equations.
\newblock {\em Ann. Probab.}, 39(3):779--856, 2011.

\bibitem{GO2}
A.~Gloria and F.~Otto.
\newblock An optimal error estimate in stochastic homogenization of discrete
  elliptic equations.
\newblock {\em Ann. Appl. Probab.}, 22(1):1--28, 2012.

\bibitem{Gu-Bal-12}
Y.~Gu and G.~Bal.
\newblock Random homogenization and convergence to integrals with respect to
  the {R}osenblatt process.
\newblock {\em J. Differential Equations}, 253(4):1069--1087, 2012.

\bibitem{GuM}
Y.~Gu and J.-C. Mourrat.
\newblock Scaling limit of fluctuations in stochastic homogenization.
\newblock {\em Multiscale Model. Simul.}, 14(1):452--481, 2016.

\bibitem{HS-94}
B.~Helffer and J.~Sj\"ostrand.
\newblock On the correlation for {K}ac-like models in the convex case.
\newblock {\em J. Stat. Phys.}, 74(1-2):349--409, 1994.

\bibitem{Kozlov-79}
S.~M. Kozlov.
\newblock The averaging of random operators.
\newblock {\em Mat. Sb. (N.S.)}, 109(151)(2):188--202, 327, 1979.

\bibitem{LNZH-17}
A.~Lechiheb, I.~Nourdin, G.~Zheng, and E.~Haouala.
\newblock {Convergence of random oscillatory integrals in the presence of
  long-range dependence and application to homogenization}.
\newblock {\em Probab. Math. Statist.}, 38(2):271--286, 2018.

\bibitem{LNP-15}
M.~Ledoux, I.~Nourdin, and G.~Peccati.
\newblock Stein's method, logarithmic {S}obolev and transport inequalities.
\newblock {\em Geom. Funct. Anal.}, 25(1):256--306, 2015.

\bibitem{Malliavin-97}
P.~Malliavin.
\newblock {\em Stochastic analysis}, volume 313 of {\em Grundlehren der
  Mathematischen Wissenschaften [Fundamental Principles of Mathematical
  Sciences]}.
\newblock Springer-Verlag, Berlin, 1997.

\bibitem{MaO}
D.~Marahrens and F.~Otto.
\newblock Annealed estimates on the {G}reen's function.
\newblock {\em Probab. Theory Related Fields}, 163(3-4):527--573, 2015.

\bibitem{MN}
J.-C. Mourrat and J.~Nolen.
\newblock Scaling limit of the corrector in stochastic homogenization.
\newblock {\em Ann. Appl. Probab.}, 27(2):944--959, 2017.

\bibitem{MO}
J.-C. Mourrat and F.~Otto.
\newblock Correlation structure of the corrector in stochastic homogenization.
\newblock {\em Ann. Probab.}, 44(5):3207--3233, 2016.

\bibitem{NP-08}
I.~Nourdin and G.~Peccati.
\newblock Stein's method on {W}iener chaos.
\newblock {\em Probab. Theory Related Fields}, 145(1-2):75--118, 2009.

\bibitem{NP-book}
I.~Nourdin and G.~Peccati.
\newblock {\em Normal approximations with {M}alliavin calculus. From Stein's
  method to universality}, volume 192 of {\em Cambridge Tracts in Mathematics}.
\newblock Cambridge University Press, Cambridge, 2012.

\bibitem{NPR-09}
I.~Nourdin, G.~Peccati, and G.~Reinert.
\newblock Second order {P}oincar\'e inequalities and {CLT}s on {W}iener space.
\newblock {\em J. Funct. Anal.}, 257(2):593--609, 2009.

\bibitem{Nualart}
D.~Nualart.
\newblock {\em {T}he {M}alliavin calculus and related topics}.
\newblock Springer-Verlag, Berlin, second edition, 2006.

\bibitem{PapaVara}
G.~C. Papanicolaou and S.~R.~S. Varadhan.
\newblock Boundary value problems with rapidly oscillating random coefficients.
\newblock In {\em Random fields, {V}ol. {I}, {II} ({E}sztergom, 1979)},
  volume~27 of {\em Colloq. Math. Soc. J\'anos Bolyai}, pages 835--873.
  North-Holland, Amsterdam, 1981.

\bibitem{Shen-07}
Z.~Shen.
\newblock The {$L^p$} boundary value problems on {L}ipschitz domains.
\newblock {\em Adv. Math.}, 216:212--254, 2007.

\bibitem{Taqqu}
M.~S. Taqqu.
\newblock Convergence of integrated processes of arbitrary {H}ermite rank.
\newblock {\em Z. Wahrsch. Verw. Gebiete}, 50(1):53--83, 1979.

\end{thebibliography}

\def\cprime{$'$} \def\cprime{$'$} \def\cprime{$'$}

\end{document}